\documentclass[12pt]{article}

\usepackage[dvips]{graphics}

\usepackage{graphicx}

\usepackage{amsfonts}
\usepackage{amsmath,amsthm,amssymb}

\begin{document}

\def\bA{\mathbb A}
\def\bR{\mathbb R}
\def\bZ{\mathbb Z}
\def\bQ{\mathbb Q}
\def\bC{\mathbb C}
\def\bN{\mathbb N}
\def\cO{\mathcal O}

\newcommand\norm[1]{\left\|#1\right\|}
\newcommand\biggnorm[1]{\biggl\|#1\biggr\|}
\newcommand\abs[1]{\left|#1\right|}
\newcommand\bigabs[1]{\bigl|#1\bigr|}
\newcommand\biggabs[1]{\biggl|#1\biggr|}
\newcommand\inn[1]{\left\langle #1 \right\rangle}
\newcommand\set[1]{\left\{{#1}\right\}}
\renewcommand\Bar[1]{\overline{#1}}

\input epsf

\def\bcw{\mathbin{\bigcirc\mkern-15mu\wedge}}

\def\r#1{{\mathop{#1}\limits^\circ}}
\def\ring{\r}

\def\nint{\mathbin{\int\mkern-18mu\diagup \;} }

\newcommand{\namelistlabel}[1] {\mbox{#1}\hfil}
\newenvironment{namelist}[1]{%
\begin{list}{}
{\let\makelabel\namelistlabel
\settowidth{\labelwidth}{#1}
\setlength{\leftmargin}{1.1\labelwidth}}
}{%
\end{list}}

\newtheorem{lem}{Lemma}
\newtheorem{thm}{Theorem}

\newtheorem{Lem}{Lemma}[section]
\newtheorem{Thm}{Theorem}[section]

\newtheorem{cor}{Corollary}[section] 

\newtheorem{Cor}[Lem]{Corollary}

\newtheorem{Prop}[Lem]{Proposition}

\newtheorem{prop}{Proposition}[section]

\newtheorem{Pro}[Thm]{Proposition}

\newtheorem{co}[Lem]{Corollary}

\newtheorem{corr}[Thm]{Corollary}

\newtheorem{C}[Lem]{Claim}

\newtheorem{Def}{Definition}[section]

\newtheorem{remark}{Remark}[section]

\pagestyle{myheadings}
\markboth{\underline{\sc Prime and Almost Prime Integral Points on
Principal Homogeneous Spaces }}
{\underline{\sc Prime and Almost Prime Integral Points on Principal
Homogeneous Spaces }}

\parskip=10pt

\baselineskip=13pt

\thispagestyle{empty}
\begin{center}
{\large \bf \sf Prime and Almost Prime Integral Points on Principal
Homogeneous Spaces }
\end{center}

\begin{center}
{\large \bf \sf by}
\end{center}

\begin{center}
\begin{tabular}{llcll}
{\Large \sf Amos Nevo} && and &&  {\Large \sf Peter Sarnak} \\ 

{\large Technion, Haifa} &&& & {\large Princeton University} \\ 
&&&& {\large \& Institute for Advanced Study} 
\end{tabular}

\vspace{.15in}
\hrule
\end{center}

\vspace{.10in}
\begin{center}
November 9, 2008. 
\end{center}

%

\vspace{.5in}

\begin{abstract}
We develop the affine sieve in the context of orbits of congruence
subgroups of semi-simple groups acting linearly on affine space.  In
particular we give effective bounds for the saturation numbers for
points on such orbits at which the value of a given polynomial has few
prime factors.  In many cases these bounds are of the same quality as
what is known in the classical case of a polynomial in one variable and
the orbit is the integers.  When the orbit is the set of integral
matrices of a fixed determinant we obtain a sharp result for the saturation number,  and thus establish the Zariski density of matrices all of whose entries are prime numbers.  Among
the key tools used are explicit approximations to the generalized
Ramanujan conjectures for such groups and sharp and uniform counting of
points on such orbits when ordered by various norms.
\end{abstract}

\vspace{.5in}


\underline{\sf Acknowledgements :}

\noindent A. N. was supported by the Institute for Advanced Study, Princeton and ISF grant 975/05. 

\noindent P. S.  was supported by an NSF grant and BSF grant 2006254.

\newpage
\noindent
\underline{\large \sf Section 1.}

\medskip
\noindent
1.A: \ \underline{\sf Prime and almost prime integral matrices of a given
determinant.}

\medskip
For integers $n \geq 2$ and $m \ne 0$ let $V_{m, n } ( \bZ )$ or
$\cO^{m,n}$ denote the set of $n \times n$ integral matrices of
determinant equal to $m$.  We study the points $x \in \cO^{m , n}$
for which $f ( x ) = \mathop{\Pi}\limits_{1 \, \leq \, i \, \leq \, j \, 
\leq \, n}\hspace{-.20in} x_{ij}$, 
or more generally any $f \in \bQ [ x_{ij}]$ which is
integral on $\cO^{m , n}$ has few (or fewest possible) prime factors.
In general, given a set $\cO$ of integer points in $\mbox{Mat}_{n\times n}$ and $d \ge 1$ let $\cO_d$
denote the reduction of $\cO$ in $\mbox{Mat}_{n\times n}( \bZ / d \bZ)$.  
We say $f$ is weakly primitive for $\cO$ if $\mbox{gcd  }\{ f ( x ): x \in \cO \}$ is $1$.  If $f$
is not weakly primitive then $\frac{1}{N} f$ is, where $N = \mbox{gcd } f ( \cO)$ and
we can represent any weakly primitive $f$ as $\frac{1}{N} g$ with $g \in \bZ [
x_{ij}]$ and $N =\mbox{gcd } g ( \cO)$.  

Define the saturation number
$r_0$ of the pair $( \cO^{m , n}, f )$ to be the least $r$ such that the
set of $x \in \cO^{m , n}$ for which $f ( x )$ has at most $r$ prime
factors, is Zariski dense in the affine variety $V_{m , n} = \{ x \in \mbox{Mat}_{n \times n}: \det x = m \} = Z c \ell ( \cO^{m , n})$ 
(we denote by $Zc\ell$ the operation of taking Zariski closure in affine
space, also see [Sa2] for a further discussion and motivation for this
set up).  It turns out that $r_0$ is finite though this is by no means
obvious.  The coordinate ring $\bQ [ x_{ij}]/(\det (x_{ij} ) - m )$ is a
unique factorization domain [San] and we factor $f$ into $t=t(f)$ irreducibles
$f_1 f_2 \ldots f_t$ in this ring.  We assume that the $f_{j}$'s are
distinct and for simplicity that they are irreducible in 
$\bar{\bQ} [ x_{ij}] / ( \det x_{ij} - m )$.  
It is clear that $r_0 ( \cO^{m , n} , f)
\geq t$ and if $f$ and the $f_j$'s have integer coefficients, that $r_0$
$(\cO^{m , n }, f) = t$ {\sf iff} the set of $x \in \cO^{m , n}$ for
which $f_j ( x )$ are all prime, is Zariski dense in $V_{m , n}$.  The
general local to global conjectures in [B-G-S2] when applied to the pair
$(V_{m , n} ( \bZ ), f )$ assert that $r_0 ( V_{m , n} ( \bZ ) , f ) =
t$.  In the case that $f$ and $f_j$ are in $\bZ [ x_{ij}]$ we even 
expect a "prime number theorem" type of asymptotics as follows:

\medskip
Let $| \ \ |$ be any norm on the linear space $\mbox{Mat}_{n \times n} ( \bR )$
and for $T \geq 1$ set 

\begin{equation}
N_{m , n } ( T ) \, = \, | \{
x \in \cO^{m , n } : | x | \, \leq \, T \} | \, .
\end{equation}

\noindent
It is known [D-R-S], [Ma], [G-N1] that 

\begin{equation}
N_{m , n } ( T ) \, \sim \, c ( \cO^{m , n} ) \, T^{n^2 - n} \
\mbox{as} \ T \longrightarrow \infty \, .
\end{equation}

\noindent
Here $c$ is positive constant and is given as a product of local
densities associated with $\cO^{m , n}$.  Let

\begin{equation}
\pi_{m, n , f }  ( T ) \, = \, | \{
x \in \cO^{m , n } : \, | x | \, \leq \, T \ \mbox{and} \ f_j ( x ) \
\mbox{is prime for}   \  j \, = \, 1 , 2 , \ldots, t \} | \, .  
\end{equation}

\noindent
Our conjectured asymptotics for $\pi_{m.n.f}$ is then as follows : 

\begin{equation}
\pi_{m , n , f} ( T ) \, \sim \,
\frac{ c \, ( {m , n , f } ) \, N_{m , n } ( T )}
{(\log T)^{t(f)}} \
\mbox{as} \ T \longrightarrow \infty \,,
\end{equation}

\noindent
where for an orbit $\cO$ we set 

\begin{equation}
c ( \cO , f ) \, = \, c_\infty ( \cO , f ) \, \prod\limits_{p  \, <
\, \infty} \;
\left[
\left(
1 - \, 
\frac{| \cO^f_p |}{|\cO_p|} \right) \,
\left( 1 \, + \, \frac{t(f)}{p} \right) \right] \,,
\end{equation}

\noindent
and $\cO^f_p$ is the subset of 
$\cO_p$ at which
$f ( x ) = 0 \, ( {\sf mod} \, p)$ while the positive archimedian factor
$c_\infty ( \cO, f )$ is a bit more complicated to describe.
We will see that the product of local densities in (1.5) converges
absolutely and each factor is non-zero since we are assuming that $f$ is
primitive.

\medskip
The main tool that we develop in this paper is an affine linear sieve
for homogeneous spaces and as in the more familiar classical
one-variable sieve [H-R], our main results are upper bounds which are
sharp up to a multiplicative constant for $\pi_{\cO , f} ( T )$ and
lower bounds, which are also sharp up to a constant factor, for points
$x \in \cO$ for which $f$ has at most a fixed number of large prime
factors (``almost primes''). 

In particular,  for the set $ \cO^{m , n}$ of integral $n\times n$ matrices of determinant $m$  the upper bound is given by 
\begin{Thm}
Let $\cO^{m, n}$ be as above and $f \in \bZ [ x_{ij}]$ be weakly primitive with
$t(f)$ irreducible factors in $\bQ [ x_{i j}] / ( \det \, x_{ij} - m )$. Then

\[
{\pi}_{m , n , f} ( T ) \ll \, 
\frac{N_{m , n}
( T )}{( \log T )^{t(f)}} \, ,
\]

\noindent
the implied constant depending explicitly on $m , n$ and $f$.

\medskip
\end{Thm}
The lower bound  is given by 
\begin{Thm}
Let $\cO^{m , n }$ be as above and let $f \in \bQ [ x_{i j} ]$ weakly primitive
and taking integer values on $\cO^{m , n}$.  Assume that $f$ has $t(f)$
distinct irreducible factors in both $\bQ [ V_{m , n }]$ and $\bar{\bQ} [
V_{m , n}]$ and let $r  > 18.t(f).n_e^3.\mbox{deg }( f )$.  Then 

\begin{equation}
\{x \in \cO^{m , n}: | x | \leq T  \ \mbox{and} \  
f ( x ) \ \mbox{has at most} \ r \ \mbox{prime factors} \}
\gg \; \frac{N_{m , n} ( T )}
{( \log T )^{t(f)}}
\, ,
\end{equation}

\noindent
Here $n_e$ is the least even integer $\ge (n-1)$, and again the implied positive constant depends on $m , n$ and $f$.
\end{Thm}

\medskip
\begin{corr}
Under the assumptions in Theorem 1.2, the saturation number satisfies the upper bound  
$r_0 ( V_{m , n} (  \bZ ) , f ) \, \leq \,  \, 18 . t(f) . n_e^3 . \mbox{ deg }(f)+1$. 
Namely the set of $x \in V_{m , n } ( \bZ)$ for which
$f ( x )$ has at most $r_0$ prime factors, is Zariski dense in $V_{m ,
n}$. 
\end{corr}

\medskip
In the case that $f ( x ) = \prod\limits_{1 \, \leq \, i \, \leq
\, j \, \leq \, n} \hspace{-.20in} x_{ij}$ Corollary 1.3 can be sharpened
considerably.  Exploiting the linearity of the determinant form in the
rows and columns of a matrix,  we use the method of Vinogradov [Vi] (see
[Va]) for handling one linear equation in 3 or more prime variables to
show that we can make all coordinates of the matrix simultaneously prime
as long as there is no local obstruction.  We have

\medskip
\begin{Thm}
$f ( x ) \, = \,\prod\limits_{1 \, \leq \, i \, \leq \, j \,
\leq \, n} \hspace{-.20in} x_{ij}$ is weakly primitive for $V_{m , n} ( \bZ)$ {\sf iff} $m
\equiv 0(\mbox{mod} \, 2^{n-1})$, and if this is the case and $n \geq 3$ then $r_0 (V_{m , n} ( \bZ ) , f ) =
n^2$.  That is for $n \geq 3$ the set of $x \in V_{m , n} ( \bZ )$ for
which each $x_{ij}$ is prime, is Zariski dense in $V_{m , n}$ {\sf iff}
$m \equiv 0 ( 2^{n - 1})$.
\end{Thm}


\medskip
\noindent
\underline{\bf Remarks 1.5}
\begin{namelist}{xxxx}
\item[1)] The proof of Theorem 1.4 provides a lower bound for
${\pi}_{m , n , f} ( T )$ which is 
a power of $T$ but not the one expected in Conjecture (1.4).  
\item[2)] Theorem 1.4 should hold when $n = 2$ but
Vinogradov's methods don not apply in this binary case.  For $m$ fixed the
infinitude  of $x_{i j}$ satisfying $x_{11} x_{22} - x_{12} x_{21} = 2
m$ and $x_{ij}$ prime, is apparently not known.  Recent work of
Goldston-Graham-Pintz and Yildrim [G-G-P-Y] on small differences of
numbers which are products of exactly two primes, shows that the desired
set is infinite for at least one number $m$ in $\{ 1 , 2 , 3 \}$.
\item[3)] An immediate improvement to the value of $r_0$ arises by choosing 
the norm to be invariant under the rotation group. This improvement, together with much more 
significant ones, will be considered systematically in \S 6. 
\end{namelist}

\medskip
\noindent
1.B: \ \underline{\sf Prime and almost prime points on principal homogeneous
spaces.}

\medskip

Theorem 1.1 and 1.2 are special cases of more general results which are
concerned with finding points on an orbit $\cO$ of $v \in \bZ^n$ under a
subgroup $\Gamma$ of $GL_n ( \bZ)$ at which a given polynomial $f \in
\bQ [ x_1 , \ldots , x_n]$ which is integral on $\cO$, has few prime
factors.  The approach is based on the ``affine linear sieve''
introduced recently in [B-G-S2].  Our purpose here is to specialize to
$\Gamma$ being a congruence subgroup of an algebraically simply connected  semisimple linear algebraic group $\mathbb{G}
\subset GL_n$ defined over $\bQ$ and for which the stabilizer of $v$ in
$\mathbb{G}$ is trivial.  This allows us to make use of the well developed
analytic methods [D-R-S], [G-N1] for counting points in such
orbits in a big Euclidean ball, as well as the strong bounds towards the
general Ramanujan Conjectures that are known from the theory
of automorphic forms (see [Cl], [Sa1]).  Our restriction to principal
homogeneous spaces (i.e. the stabilizer of $v$ being trivial) or $\mathbb{G}$
being algebraically simply connected are not serious ones as far as the production of
a Zariski dense set of points $x$ on an orbit at which $f ( x )$ has few
prime factors.  As explained in [B-G-S2] the dominant $\bQ$-morphisms
from $\mathbb{G}$ to an orbit $\mathbb{G}/\mathbb{H}$ and from the simply connected cover
$\tilde{\mathbb{G}}$ of $\mathbb{G}$, to $\mathbb{G}$, reduce the basic saturation problem for orbits of
more general congruence groups to the cases that we consider in this
paper.  

\medskip
We describe our results in more detail.  Let $\mathbb{G} \subset GL_n ( \bC )$ be
a connected and algebraically simply connected semi-simple algebraic matrix group defined over
$\bQ$.  Let $G ( \bQ)$ be its rational points and $\Gamma = \mathbb{G} ( \bZ) = \mathbb{G}
( \bQ ) \cap GL_n ( \bZ)$ its integral points.  Fix $v \in \bZ^n$ and
let $V = \mathbb{G} . v$ be the corresponding orbit which we assume is
Zariski closed in $A^n$.  Since $\mathbb{G}$ is algebraically connected and the stabilizer of
$v$ is assumed to be trivial, $V$ is an absolutely irreducible affine
variety defined over $\bQ$ and has dimension equal to $\dim \mathbb{G}$.  The
ring of $\mathbb{G}$-invariants for the action of $\mathbb{G}$ on the $n$-dimensional
space $A^n$ separates the closed $\mathbb{G}$-orbits ([B-HC]).  We can choose
generators $h_1 , h_2 , \ldots , h_\nu$ in $\bQ [ x_1 , \ldots , x_n ]$ of
this ring so that $V$ is given by

\begin{equation}
V: h_j ( x ) \, = \, \lambda_j , \, j \, = \, 1 , \ldots \nu \, , \
\mbox{with} \ \lambda_j \in \bQ \, .
\end{equation}

\noindent
Let $\cO = \Gamma . v$ be the $\Gamma$-orbit of $v$ in $\bZ^n$.
According to [Bo] $\cO$ is Zariski dense in $V$.  The coordinate ring of
$V$, $\bQ [ x _1,\ldots, x_n] / (h_1 - \lambda_1 , \ldots , h_\nu - \lambda_j)$ is a
unique factorization domain [San].  Hence an $f \in \bQ [ x_1 , \ldots ,
x_n]$ factors into irreducibles $f = f_1 f_2 \cdots f_t$ in this ring.
We assume that these $f_j$'s are distinct and that they are irreducible
in $\bar{\bQ} [ x_1 , \ldots x_n] / (h_1 - \lambda_1 , \ldots , h_\nu -
\lambda_\nu)$, that $f$ takes integer values on $\cO$ and that it is
$\cO$-weakly primitive.  The saturation number $r_0 ( \cO , f )$ of the pair
$(\cO, f )$ is the least $r$ such that the set of $x \in \cO$ for which
$f ( x )$ has at most $r$ prime factors is Zariski dense in $V ( = {\sf
Z c l} ( \cO ))$.

\medskip
To order the elements of $\cO$ we use the following ``height''
functions.  Let $\parallel  \, \parallel$ be any norm on the linear space
$\mbox{Mat}_{n \times n} ( \bR )$.  For $T > 0$ set 

\begin{equation}
\cO ( T ) \, = \, \{
\gamma v: \, \parallel \gamma \parallel \, \leq \, T \, , \, \gamma \in
\Gamma \}
\end{equation}  

\noindent
(this depends on $v$ but in an insignificant way).

In many interesting cases these sets $\cO ( T )$ can be described as $\{ x
\in \cO: \, | x | \, \leq \, T \}$ where $ | \ |$ is a norm on $\bR^n$,
but this is not true in general.  The main term of the asymptotics for
$N_\cO ( T ) = | \cO ( T ) |$ is known for any norm ([G-W], [Ma], [GN1]) and
takes the form 

\begin{equation}
N_\cO ( T ) \, \sim \, c_1 ( \cO ) T^a ( \log T )^b
\end{equation}

\noindent
with $c_1 ( \cO ) > 0$ and $a > 0$, and $b \in \bN$ a non-negative integer.  The
numbers $a$ and $b$ are given explicitly in terms of the data in
Theorem 3.1 below (see the discussion following that theorem) and they are
independent of $\parallel \ \parallel$. 

\medskip
\setcounter{Thm}{5}
\begin{Thm}
Let $\cO$ and $f$ be as above and assume that $f_j , j = 1 , \ldots ,
t(f)$, are integral on $\cO$, then for $T \geq 2$,

\[
| \{
x \in \cO ( T ) : \, f_j ( x ) \ \mbox{is prime for} \ j \, = \, 1 , 2 ,
\ldots , t (f)\} | \, \ll \, 
\frac{N_\cO ( T )}{( \log T )^{t(f)}}
\]

\noindent 
the implied constant depending on $f$ and $\cO$.
\end{Thm}

\medskip
\begin{Thm}
Let $\cO$ and $f$ be as above and let $r  > ( 9 t(f)\,  ( 1 + \dim G
)^2 \cdot \,2n_e(\Gamma) \deg (f) ) / a$, where $n_e(\Gamma )$ is the
integer defined following Theorem 3.3.  Then as $T \longrightarrow
\infty$

\[
| \{ x \in \cO (T ): f ( x ) \ 
\mbox{has at most} \ r \ \mbox{prime factors} \}
| \, \gg \, \frac{N_\cO ( T )}{(\log T )^{t(f)}} \, ,
\]

\noindent
the implied constant depending on $f$ and $\cO$.
\end{Thm}

\begin{corr}
Let $\cO $ and $f$ be as above.  Then

\[
r_0 ( \cO , f ) \, \leq \, \frac{ 9 t(f) \, ( 1 + \, \dim G )^2 \cdot 2n_e(\Gamma)\deg (f)}{ a} +1\, .
\]

\end{corr}

\noindent
\underline{\bf Remark 1.9}
\begin{namelist}{xxxxx}
\item[(i)] The integer $n_e (\Gamma )$ is at least 1 and is determined by 
the extent to which the representation spaces $L^2_0(G)/\Gamma(q))$ weakly contain non-tempered irreducible representations of $G=\mathbb{G}(\mathbb{R})$.  
Here $L^2_0(G/\Gamma)$ is the space of functions with zero integral, and $\Gamma(q)$ any congruence subgroup of $\Gamma$. 
The non-temperedness is measured by 
the infimum over all $p > 0$ for which the representation space contains a dense subspace 
of matrix coefficients belonging to $L^p(G)$. Thus $n_e(\Gamma)$  is directly
connected to the generalized Ramanujan Conjectures for 
$\mathbb{G} ( \mathbb{A} ) \slash \mathbb{G} ( \mathbb{Q})$ [Sa1].
\item[(ii)] Theorems 1.1, 1.2 and Corollary 1.3 are connected to the
general Theorems 1.6, 1.7 and 1.8 as follows; $\Gamma = SL_n ( \bZ)$
acts on $V_{n , m} ( \bZ)$ by left multiplication.  This action has
finitely many orbits.  So with $\mathbb{G} = SL_n \subset GL_M$, $M = n^2$ and
this action, $V_{n, m} = \mathbb{G} . v$ where $\det ( v ) = m$.  Theorem 1.1
then follows by applying Theorem 1.6 to each orbit separately.  For
Theorem 1.2, the difference between the individual orbit $\cO$ of $SL_n
( \bZ)$ and all of $V_{m , n} ( \bZ)$ raises the issue of the weak primitivity
of $f$ on $V_{m , n} ( \bZ)$.  So one needs to globalize the argument as
is explained in Section 4.
\end{namelist}

\medskip
The values of $r$ in Theorem 1.7 and Corollary 1.8 are by no means
optimal.  There are various places where the analysis can be modified to
give far better bounds.  Firstly by using smooth positive weights 
instead of the sharp cutoff counting function in (2.9) and Theorem 3.2,
we can improve the level of distribution $\tau$ in (4.15).  We carry
this out in Section 6 for the case that $\Gamma$ is co-compact in $G=\mathbb{G} (
\bR)$.  In Theorem 6.1 we obtain the sharpest possible remainder for
such smooth sums in terms of bounds towards the Ramanujan Conjectures.
This leads to the improvement in the value for $r$ that is given in 
(6.5).  A further improvement is gotten by using a weighted sieve ([H-R],
[D-H]) rather than the simple sieve from Section 2.  This leads to the
improved value for $r$ given in (6.15).  Finally there are cases
such as the following for which very strong bounds on the spectrum of
$L^2_0 (G  \slash \Gamma ( q ))$ are known and which result in quite
good values for $r_0$.  Let $D$ be a division algebra over $\bQ$ of
degree $n$ (which for technical reasons we assume is prime) and for
which $D \otimes \bR \cong \mbox{Mat}_{n \times n} ( \bR)$.  Then $D ( \bQ )$
has dimension $M = n^2$ over $\bQ$ and choosing a basis gives a
$\bZ$-structure for $D$, that is $M$ coordinates $x_{ij}$.  Consider 
 the reduced norm and let $\mathbb{G}$ be the linear algebraic group of
elements of reduced norm $1$.  Let $\Gamma = \mathbb{G} ( \bZ ) \subset GL_M$
where the action is by multiplication on the left.  Let $\cO = \Gamma
. v$ where $v \in D ( \bZ)$ has reduced norm $ m \ne 0$,  and let $f \in
\bQ [ x_{ij} ]$ be $\cO$ integral and weakly primitive.  Then all the improvements
mentioned above apply to the pair $( \cO , f )$ and Theorem 1.7 and
Corollary 1.8 apply with the following value for $r_0$ (see Theorem 6.3):

\begin{equation}
r_0 ( \cO , f ) \, \leq \, 6 \, \deg(f) \, + \, t (f)\, \log  t(f)  \, .
\end{equation}

\noindent
This bound is independent of dimension and it is of the same quality and
shape, in terms of dependence on the degree of $f$ and the number of its
irreducible factors, as what is known for $r$-almost primes for values
of $f ( x )$ in the classical case of one variable ([H-R]).

\medskip
Uniform bounds such as those in (1.10) are useful when combined with
$\bQ$-morphisms.  Let $\phi: \mathbb{G} \longrightarrow A^k$ be a $\bQ$-morphism
of affine varieties for which $\cO = \phi ( \mathbb{G} ( \bZ )) \subset \bZ^k$.
Then if $f \in \bQ [ x_1 , \dots , x_k]$ is $\cO$-integral and weakly primitive
then $\phi^\ast ( f ) = f \circ \phi$ is $\mathbb{G} ( \bZ)$-integral and
weakly primitive.  Moreover, $r_0 ( \cO , f ) \leq r_0 ( \mathbb{G} ( \bZ) , \phi^\ast (
f ))$.  If $\cO$ is part of a larger set of integral points that can be
swept out by varying $\phi$ suitably, then the uniformity allows one to
give bounds for saturation numbers for the larger set.  This of course
applies also with $\mathbb{G} = A^1$ in which case one can apply the classical
$1$-variable sieve.  For example, Corollary 1.3 can be approached by
this more elementary method.  Let $y \in V_{m , n} ( \bZ)$ and let
$\phi_y: A^1 \longrightarrow V_{m , n}$ be the morphism

\[
\phi_y: \; x \longrightarrow \; \left[
\begin{array}{llllll}
1 & 0 & 0 & 0 & x \\
& 1 &&&0 \\
&&&&0 \\
& 0&&&1
\end{array}
\right]
y \, .
\]

\noindent
Then $\cO = \phi_y ( \bZ) \subset V_{m , n} ( \bZ)$.  Apply the
classical $1$-variable results about almost primes to the pair $( \bZ,
\phi^\ast ( f ) )$.  For a generic $y$ the bound for $r_0$ depends only
on $t$ and $d$ and the set of such $y$'s is Zariski dense in $V_{m ,
n}$.  In this way one can establish Corollary 1.3 with $r_0$ comparable to (1.10) above. This approach is possible whenever $G$ has $\bQ$-unipotent elements.  
However, in the opposite case when $\mathbb{G} ( \bR) / \mathbb{G} (
\bZ)$ is compact, such as in the division algebra examples above, there
are no such $\mathbb{Q}$-rational parametric affine curves in $G=\mathbb{G}(\mathbb{R})$ and the general
affine linear sieve developed in this paper is the only approach that we
know of to obtain Corollary 1.8.  In [L-S] this technique is developed for
anisotropic quadratics in 3-variables.

\medskip
\noindent
\underline{\large \sf {Section 2. \ The Combinatorial Sieve}}

To begin with we make use of a simple version of the combinatorial
sieve, see [IK, \S 6.1-6.4] and [H-R Theorem 7.4].  Later we will use
more sophisticated versions.  Our formulation is tailored to the
applications.

\medskip
Let $\mathcal{A} \, = \, (a_k)_{k \geq 1}$ be a finite sequence of nonnegative
numbers.  Denote by $X$ the sum

\setcounter{section}{2}
\setcounter{equation}{0}
\begin{equation}
\displaystyle{\sum\limits_{k}} a_k \, = \, X \, .
\end{equation}
We think of $X$ as large, tending to infinity as the number of elements
of $\mathcal{A}$ increases.  Fix a finite set $B$ of ``ramified'' primes which for
the most part will be the empty set.  For $z$ a large parameter (in
applications $z$ will be a small power of $X$) set
\begin{equation}
P \, = \, P_{z , B} \, = \, 
\displaystyle \prod\limits_{p \, \leq \, z
\atop{p \, \notin \, B}} \, p \, .
\end{equation}
Under suitable assumptions about the sums of the $a_k$'s in progressions
with moderately large moduli, the sieve gives upper and lower
estimates which are of the same order of magnitude, for the sums of $\mathcal{A}$
over $k$'s which remain after sifting out numbers with prime factors in
$P$.

\medskip
More precisely let
\begin{equation}
S ( \mathcal{A} , P ) \, := \, \displaystyle{\sum\limits_{(k , P ) = 1}} \, a_k \,
.
\end{equation}
The assumptions for the sums over progressions that we make are as
follows: 
\begin{enumerate}
\item[(A$_0$)] For $d$ square free and having no prime factors in $B$,
$(d < X )$, we assume that the sum over multiples of $d$ takes the form
\begin{equation}
\displaystyle{\sum\limits_{k \, \equiv \, 0 ( d )}} \, a_k \, = \,
\frac{\rho ( d )}{d} \, X \, + \, R ( \mathcal{A} , d )
\end{equation}
where $\rho ( d )$ is multiplicative in $d$ and for $p \notin B$ 
\begin{equation}
0 \, \leq \, \frac{\rho ( p )}{p } \, \leq \, 1 \, - \, \frac{1}{c_1} \,
.
\end{equation}
The understanding being that $\frac{\rho ( d )}{d} X$ is the main term
in (2.4) and $R ( \mathcal{A} , d )$ is smaller, at least on average.
\item[(A$_1$)] $\mathcal{A}$ has level distribution $D = D ( X )$, that is for
some $\epsilon > 0$ there is $C_\epsilon < \infty$ such that 
\begin{equation}
\displaystyle{\sum\limits_{d \, \leq \, D}} \, | R ( d , \mathcal{A} ) | \, \leq
\, C_\epsilon \, X^{1 - \epsilon} \, .
\end{equation}
If this holds with $D = X^\tau$ we say that the level distribution is
$\tau$.
\item[(A$_2$)] $\mathcal{A}$ has sieve dimension $t$, that is there is
$C_2$ fixed such that 
\begin{equation}
- C_2 \, \leq \, \displaystyle{\sum\limits_{(p , B ) \, = 1 \atop{w \, \leq
  \, p \, \leq \, z}}} \; \frac{\rho ( p ) \, \log p}{p} \, - t \, \log
\frac{z}{w} \, \leq \, C_2\,\,,\,\,\,\, \hspace{.5em} \mbox{for} \hspace{.5em} 2 \,
\leq \, w \, \leq \, z \, .
\end{equation} 
\end{enumerate}

Assuming (A$_0$), (A$_1$) and (A$_2$) the simple combinatorial sieve
that we use asserts that for $s > 9t$ and $z = D^{1/s}$ and $X$ large
enough 
\begin{equation}
\frac{X}{( \log X )^t} \, \ll \, S ( \mathcal{A}, P )  \, \ll \, 
\frac{X}{ ( \log X )^t}
\end{equation}
where the implied constants depend explicitly on $t , \epsilon, C_1 ,
C_2$ and $C_\epsilon$ (all of which are fixed).

For our application $V \subset A^n$ is a principal homogeneous
space for $G$ and $\mathcal{O} = G ( \mathbb{Z} ) . v$ is an orbit
of integral points in $V$.  $f \in \mathbb{Q} [ x_1 , \cdots x_n]$ 
is integral on $\mathcal{O}$ and $\parallel \ \parallel$ is the norm
described in Section 1.  For $k \geq 0$, we let $a_k = a_k ( T )$ be given
by 
\begin{equation}
a_k ( T ) \, : = \, \displaystyle{\sum\limits_{\parallel \gamma
\parallel \, \leq \, T \atop{| f ( \gamma v ) | \, = \, k }}} 1\,\,\,\, .
\end{equation}

Under the assumptions (2.4), (2.6) and (2.7) on the level distribution
it follows that 
\begin{equation}
a_0 ( T ) \, \ll \, \frac{1}{D} \, X \log X
\end{equation}
where $D$ is the level.  Hence for our purposes we may include $k = 0$
into the sieve analysis without affecting (2.8).

A large part of the paper is concerned with verifying (A$_0$) (A$_1$)
and (A$_2$) for this sequence and determining an admissible level of
distribution.

\medskip
\noindent
\underline{\large \sf Section 3. \ Uniform Lattice Point Count}

In this and the next section we will identify the main term in the
asymptotics of\break ${\sum_{k \equiv 0 ( d )}}$ $ a_k ( T )$
(condition (A$_0$)), as well as estimate the level of distribution $D$
(condition (A$_1$)).  In \S4.1 we will establish the multiplicativity of
$\rho ( d ) \slash d$, the coefficient of the main term, concluding the
proof of (A$_0$) and (A$_1$).  The most demanding part is establishing
an explicit bound on the error terms ${\sum_{d \leq
D}} \, | R( d , \mathcal{A} )|$ appearing in condition (A$_1$).  Here the basic
ingredient will be an error estimate for the lattice points counting
problem which is uniform over all cosets of all congruence groups.  This
will be established in \S3.2 and \S3.3. 

\medskip
\noindent
\underline{\sf 3.1 \ Spectral estimates}

Let $\mathbb{G} \subset GL_n ( \mathbb{C} )$ be a connected semisimple
algebraic matrix group defined over $\mathbb{Q}$, with $G = \mathbb{G} (
\mathbb{R})$ having no non-trivial compact factors.  Fix any norm on the
linear space $M_n ( \mathbb{R})$.  Let $\Gamma ( 1 ) = \mathbb{G} (
\mathbb{Z})$ be the group of integral points, which is a lattice
subgroup of $G$ [B-HC] and a subgroup of $GL_n ( \mathbb{Z})$.  Let
$\Gamma ( q ) , q \in \mathbb{N}$ denote the principal congruence
subgroup of level $q$, namely
\[
\Gamma ( q ) \, = \, \{ \gamma \in \Gamma; \, \gamma \, \equiv \, I
(\!\!\!\!\!\!\mod
q ) \} \, .
\]

We begin by stating the following volume asymptotic, established in  [G-W], [Ma] . 

\setcounter{section}{3}
\setcounter{Thm}{0}
\begin{Thm}   
Let $\mathbb{G}$ and $G$ be as above and let $ \parallel \, \parallel$ be any norm on $M_n(\mathbb{R})$.
Then there exists $a > 0$ and a non-negative integer $b$ both depending only on
$G$, and a constant $B$ depending also on the norm, such that 
\[
\lim\limits_{T \rightarrow \infty} \; \frac{{\rm vol} \{ g \in G ; \,
\parallel g \parallel \, \leq \, T \}}{B T^a ( \log T )^b} \, = \, 1 \,
.
\]
\end{Thm}

The exponent $a$ has the following simple algebraic
description [G-W][Ma].  Let $A$ denote a maximally $\mathbb{R}$-split
Cartan subgroup of $G = \mathbb{G} ( \mathbb{R})$, with Lie algebra
$\frak{a}$.  Let $\mathcal{C}$ denote the convex hull of the weights of
$\frak{a}$ associated with the representation of $G$ in $GL_n (
\mathbb{R})$.  Let $2 \rho_G$ denote the sum of all the positive roots
(counted with multiplicities) of $\frak{a}$.  Then $a$ is the unique
positive real number with the property that $2 \rho_G \slash a \in
\partial \mathcal{C}$.  The parameter $b + 1$ is a positive integer,
which is at most the $\mathbb{R}$-rank $G$, and is equal to the codimension
of a minimal face of the polyhedron $\mathcal{C}$ containing the point
$2 \rho_G \slash a$. In particular, both $a$ and $b$ are independent of the norm. 

We now state the following uniform remainder estimate for the lattice
point counting problem, which underlies our estimate of the level of
distribution.  
\begin{Thm}
Let $\mathbb{G}, G, \Gamma (q)$ and $\parallel  \parallel$ be as
above and  normalize Haar measure on $G$ so that vol$(G \slash \Gamma ( 1
)) = 1$.  Then the following uniform error estimate holds (for any $\eta > 0$) 
\[
\frac{\# \{ w \in \Gamma ( q ) y ; \parallel w \parallel \leq T \}}{{\rm
vol} \{ g \in G; \parallel g \parallel \leq T \} } \, = \, \frac{1}{[
\Gamma : \Gamma ( q ) ]} \, + \, {O}_\eta \left(
T^{- \frac{\theta}{1 + {\rm dim} G} \, + \, \eta} \right) \, ,
\]
where
\begin{itemize}
\item $\theta =\frac{a}{2n_e(G,\Gamma)}> 0$ is explicit and depends only on the bounds towards
the generalized Ramanujan conjecture for the homogeneous spaces $G
\slash \Gamma (q)$ (see Theorem 3.3 below),
\item the estimate holds uniformly over
all cosets of all congruence subgroups $\Gamma ( q )$ in $\Gamma ( 1 )$, namely 
the implied constant is independent of $q$ and $y \in \Gamma (1)$ (it depends only on $\eta$ and the chosen norm).
\end{itemize}
\end{Thm}

A general approach to the lattice point counting
problem with error estimate for $S$-algebraic groups is developed in
[GN1], based on establishing quantitative mean ergodic theorems for the Haar-uniform averages supported on the norm balls. In \S\S3.2-3.3 we follow this approach and  give a short proof of Theorem 3.2, thus establishing the uniformity of the error
term, which is crucial for our considerations.  More general results can
be found in [GN2].

Let $B_t \subset G , t \in \mathbb{R}_+$ denote the family of subsets
\[
B_t \, = \, \{ g \in G: \, \parallel g \parallel \, \leq \, e^t \} \, ,
\]
and let $\pi_{G \slash \Gamma} ( \beta_t)$ denote the following
averaging operator acting in $L^2 ( G \slash \Gamma )$:
\[
\pi_{G \slash \Gamma}(\beta_t) \, f ( x ) \, = \,
\displaystyle{\int\limits_{g \in B_t}} \, f ( g^{-1}x ) \, dm_{G } \, .
\]
The subspace $L^2_0 ( G \slash \Gamma )$ of functions of zero mean is
obviously an invariant subspace, and the representation there is denoted
by $\pi^0_{G \slash \Gamma}$.

The fundamental spectral estimate in our discussion is given by
\begin{Thm}
Let $\mathbb{G}, G, \Gamma(q)$ and $B_t$ be as above.  Then for  
$\theta =\frac{a}{2n_e(G,\Gamma)} > 0$, uniformly for every $q \in \mathbb{N}$

\[
\Big\| \pi_{G \slash \Gamma ( q )} ( \beta_t ) f \, - \, 
\int\limits_{G \slash \Gamma ( q )} \, 
f d m_{G \slash \Gamma ( q ))} 
\Big\|_{L^2 ( G \slash \Gamma ( q ))} \, \leq \, C_\eta
e^{- (\theta-\eta) t } \, 
\parallel f \parallel_{L^2 ( G \slash \Gamma ( q))}
\]
where $m_{G \slash \Gamma ( q)}$ is the $G$-invariant probability measure
on $G \slash \Gamma ( q )$.  Here $\eta > 0$ is arbitrary,  $a$ is the exponent in the rate of exponential volume growth
of $B_t$, and $n_e ( G , \Gamma)$ is the least even integer $\ge p(G,\Gamma)/2$, with $p(G,\Gamma)$ the bound towards the Ramanujan conjecture described in the proof. 
\end{Thm}
\begin{proof}
The proof of Theorem 3.3 consists of two parts.  The first part is to
note that the bounds towards the generalized Ramanujan conjecture (see
[Cl1] and [Sa1])
imply that there exists an explicit $p = p ( G , \Gamma )$ with the
property that all the ($K$-finite) matrix coefficients occurring in\break
$L^2_0 (G \slash \Gamma (q))$ are in $L^{p + \eta} ( G )$ for all
$\eta > 0$, and all $q \in \mathbb{N}$.  The second is to use this
information to give an explicit estimate for $\theta$.

For the first part, let us note that our lattice $\Gamma = \mathbb{G} (
\mathbb{Z})$ is an irreducible lattice, and as a result, in the unitary
representation of $G$ in $L^2_0 ( G \slash \Gamma )$, the matrix
coefficients decay to zero as $g \rightarrow \infty$ in $G$, namely the
representation is strongly mixing (recall that we assume $G$ has no
non-trivial compact factors).  In particular, if $H$ is an almost-simple
component of $G$, then $H$ has no invariant unit vectors in $L^2_0 ( G
\slash \Gamma (q))$.  We now divide the argument into two cases.
\begin{enumerate}
\item Assume that $H$ has property $T$.  Then there exists $p = p_H$,
such that {\it every} unitary representation of $H$ without $H$-fixed
unit vectors has its ($K$-finite) matrix coefficients in $L^{p +
\eta}$ for all $\eta > 0$ [Co].  Thus we can use as a bound for
the automorphic spectrum a bound valid for all unitary representations
of $H$.  Therefore in this case the integrability parameter $p$ above depends
only on $H$ but not on $\Gamma$.  In particular, the bound holds in $L^2_0(G/\Gamma(q))$ for all
finite-index subgroups $\Gamma(q)$, including the principal congruence
subgroups.

However, note that for some lattices $\Gamma$ one can do better; this
applies in particular to certain uniform arithmetic lattices, as will be
discussed further in \S6 below.

For a list of $p_H$ for classical simple groups $H$ with property $T$ we refer to [Li],
and for exceptional groups, to [LZ], [Oh] and [Lo-S].  Thus for example $p_{SL_n (
\mathbb{R})} = 2 ( n - 1) ( n \geq 3)$, and  $p_{Sp ( n , \mathbb{R})} = 2n
( n \geq 2)$. 
\item $G$ is defined over $\mathbb{Q}$, and so are its simple
component subgroups.  Let now $H$ be a simple algebraic
$\mathbb{Q}$-subgroup of real rank one which does not satisfy property
$T$, and $\Gamma = \mathbb{G} ( \mathbb{Z})$.  Then there still exists
$p = p ( H , \Gamma )$, such that the set of representations of $H$
obtained as restrictions from the representations of $G$ on $L^2_0 ( G
\slash \Gamma ( q ))$ have their ($K$-finite) matrix coefficients in
$L^{p + \eta} ( H )$, $\eta > 0$ for all $q \in \mathbb{N}$.  This fact is
established in most cases in [B-S], and in the missing cases by [Cl].
These results yield the
following explicit estimate.  Let $\rho_H = \frac{1}{2} ( m_1 + 2m_2)$
where $m_1$ (resp. $m_2$) is the multiplicity of the short (resp. long)
root in the root system associated with a maximal $\mathbb{R}$-split
torus in $H$.  The parameter $s$ of a non-trivial complementary
series representation $\pi_s$ that can occur in the automorphic spectrum
of $H$ is constrained to satisfy $s \leq \rho_H - \frac{1}{4}$.  Now the
volume density on $H$ in radial coordinates is comparable to exp($2
\rho_H t$),  and the decay of the spherical function $\varphi_s$ is
comparable to exp$(- \frac{1}{4} t)$, so that the matrix coefficients
are in $L^{p + \eta} ( H )$ where $p = 8 \rho_H = 4 ( m_1 + 2m_2)$. 
\end{enumerate}

Finally to conclude the first part of the proof, note that for $p = p (
G , \Gamma )$ we may take the maximum of $p ( H , \Gamma )$ as $H$
ranges over the almost-simple $\mathbb{Q}$-subgroups, since any
(normalized) matrix coefficient has absolute value bounded by 1.

The second part of the proof consists of showing how to derive an
explicit estimate for the decay of the operator norms of $\pi^0_{G
\slash \Gamma ( q )} ( \beta_t)$ from the bound on the automorphic
spectrum.  Let $p = p ( G , \Gamma )$ be the minimum value such that every strongly mixing
unitary representation weakly contained  has its ($K$-finite) matrix coefficients in $L^{p
+ \eta}(G)$ for all $\eta > 0$.  Recall that  we define $n_e = n_e (\Gamma )$
is the least even integer greater than or equal to $ p(G,\Gamma) \slash 2$.  By 
[N1, Thm. 1]  

$$
\parallel \pi^0_{G \slash \Gamma} ( \beta_t ) \parallel_{L^2_0 ( G
\slash \Gamma )} \, \leq \, \parallel \lambda_G ( \beta_t )
\parallel^{\frac{1}{n_{e}}  }_{L^2 ( G )}
$$

\noindent where $\lambda_G$ is the regular representation of $G$ on $L^2 ( G )$.
Now following [N1, Thm. 4], by the Kunze-Stein phenomenon the norm of the
convolution operator $\lambda_G ( \beta_t)$ determined by $\beta_t$ on
$L^2 ( G )$ is bounded by $C^\prime_\eta$ vol$(B_t)^{- \frac{1}{2} +
\eta}$.  Taking the volume asymptotics of $B_t$ stated in Theorem
3.1 into account, we conclude that $\theta = a / 2n_e$ gives the norm
bound stated in Theorem 3.3.

\end{proof}

\medskip
\noindent
\underline{\sf 3.2 \  Averaging operators and counting lattice points}

We now turn to a proof of Theorem 3.2, and begin by explicating the
connection between the averaging operators associated with $\beta_t$ and
counting lattice points, following the method developed in [GN1]. 

Consider the bi-$K$-invariant Riemannian metric on $G$ covering the
Riemannian metric associated with the Cartan-Killing form on the
symmetric spaces $S =G \slash K$, and let $d$ denote the distance
function.  Let vol denote the Haar measure defined by the volume form
associated with the Riemannian metric.  Define
\[
{O}_\epsilon = \{ g \in G: \, d ( g , e ) \, < \, \epsilon \}.
\]
Recall that $B_t \subset G , t \in \mathbb{R}_+$ is the family of
subsets
\[
B_t \, = \, \{ g \in G ; \parallel g \parallel \, \leq \, e^t \} \, .
\]

The sets $B_t$ enjoy the following stability and regularity properties. 

\setcounter{prop}{3}
\begin{prop}{[GN1, Thm. 3.15].} 
The family $B_t$ is admissible, namely there exists $c >
0$, $\epsilon_0 >$ and $t_0 > 0$ such that for all $t \geq t_0$ and $0 <
\epsilon < \epsilon_0$
\end{prop}
\setcounter{equation}{0}
\begin{equation}\label{adm1}
{O}_\epsilon \, . \, B_t \, . \, {O}_\epsilon
\subset B_{t + c \epsilon}, 
\end{equation} 
\begin{equation}\label{adm2}
\mbox{vol}(B_{t + \epsilon}) \, \leq \, ( 1 \, + \, c \epsilon) \, . 
\, \mbox{vol} (B_t) \, .
\end{equation}

Given the lattice $\Gamma=\Gamma(1)=G(\mathbb{Z})$, fix the unique invariant volume form vol=vol$_G$ on $G$ satisfying vol$(G/\Gamma)=1$. We denote by vol$_{G/\Gamma(q)}$ the volume form 
 induced on $G \slash \Gamma(q)$ by the volume form on $G$.
Then  vol$_{G \slash \Gamma(q)} ( G \slash \Gamma(q) )=[\Gamma:\Gamma(q)]$ is the total volume of the locally symmetric space $G/\Gamma(q)$.  We also
let $m_{G \slash \Gamma(q)}$ denote the corresponding probability measure
on $G \slash \Gamma(q)$, namely  vol$_{G \slash \Gamma(q)} \slash 
[\Gamma:\Gamma(q)] )$.  

We note that clearly $\frac{1}{2} \epsilon^{\dim G} \leq \, \mbox{vol}
({O}_\epsilon ) \, 
\leq \, 2 \epsilon^{\dim G}$, for 
$0 < \epsilon \leq  \epsilon^\prime_0$.  Denote
\[
\chi_\epsilon \, = \, \frac{\chi_{{O}_\epsilon}}
{\mbox{vol} ( {O}_\epsilon )}
\]

We now fix a congruence subgroup $\Gamma(q)\subset \Gamma=\Gamma (1)$, and define,
for every given $y \in \Gamma (1)$
\[
\phi^y_\epsilon ( g \Gamma(q) ) \, = \, 
\sum\limits_{\gamma \, \in \, \Gamma(q)} \, 
\chi_\epsilon ( g \gamma y ) \, ,
\]
so that $\phi^y_\epsilon$ is a measurable bounded function on $G \slash
\Gamma(q)$ with compact support, and 
\[
\int\limits_{G} \chi_\epsilon d \, \mbox{vol} = 1 \, , \ \mbox{and} \ 
\int\limits_{G \slash \Gamma(q)} \, 
\phi^y_\epsilon d \,  \mbox{vol}_{G \slash \Gamma(q)} 
\, = \, 1 \, ,  \ \mbox{and so} \ \int\limits_{G \slash \Gamma(q)}
\, \phi^y_\epsilon \,  dm_{G \slash \Gamma(q)} \, = \, \frac{1}{[ \Gamma: \Gamma(q)] }
\, .
\]

Clearly, for any $\delta > 0$, $h \in G$ and $t \in \mathbb{R}_+$, the
following are equivalent (for any function on $G \slash \Gamma(q)))$:
\begin{equation}
\Big| \pi_{G \slash \Gamma} ( \beta_t) \, \phi^y_\epsilon (h \Gamma(q) ) \,
- \, \frac{1}{[\Gamma:\Gamma(q)] } \Big| \, \leq \, \delta \, ,
\end{equation}
\begin{equation}
\frac{1}{ [\Gamma:\Gamma(q)] )} \, - \delta \, \leq \, \frac{1}{\mbox{vol} (
B_t)} \, \int\limits_{B_t} \, \phi^y_\epsilon ( g^{-1} h \Gamma(q) ) \, d
\,
\mbox{vol} ( g ) \, \leq \, \frac{1}{[ \Gamma:\Gamma(q)] } \, + \, \delta \, .
\end{equation}

The set where the first inequality holds will be estimated using the quantitative mean ergodic theorem. The integral in the second expression is connected to lattice points as follows:

\addtocounter{Lem}{4}
\begin{Lem} 
For every $t \geq t_0 + c \epsilon_0$, $0 < \epsilon
\leq \epsilon_0$ and for every $h \in \mathcal{O}_\epsilon$,
\end{Lem}
\[
\int\limits_{B_{t - c \epsilon}} \, \phi^y_\epsilon ( g^{-1} h \Gamma(q) )
\, d \, \mbox{vol} ( g ) \, \leq \, | B_t \cap \Gamma(q) y| \, \leq \,
\int\limits_{B_{t + c \epsilon}} \,
\phi^y_\epsilon ( g^{-1} h \Gamma(q) ) \, d \, \mbox{vol} ( g ) \, .
\]

\begin{proof}
If $\chi_\epsilon ( g^{-1} h \gamma y ) \ne 0$ for some $g \in B_{t - c
\epsilon}$, $h \in {O}_\epsilon$ and $ \gamma y \in \Gamma (q) 
y$, then by (\ref{adm1})
\[
\gamma y \in h^{-1} \, . \, B_{t - c \epsilon} \, . \, (
\mbox{supp} \; \chi_\epsilon ) \subset B_t \, .
\]
Hence,
\[
\int\limits_{B_{t - c \epsilon}} \, \phi^y_\epsilon ( g^{-1} h \Gamma(q) )
\, d \, \mbox{vol} ( g ) \, \leq \,
\sum\limits_{\gamma y \in B_t \cap \Gamma(q) y} \: \int\limits_{B_t} \,
\chi_\epsilon ( g^{-1} h \gamma y ) \, d \, \mbox{vol} ( g ) \, \leq \, | B_t
\cap \Gamma(q) y | \, .
\]

In the other direction, for $\gamma y \in B_t \cap \Gamma (q) y$ and $h
\in {O}_\epsilon$,
\[
\mbox{supp} ( g \mapsto \chi_\epsilon ( g^{-1} h \gamma y )) \, = \, h
\gamma y ( \mbox{supp} \chi_\epsilon)^{-1} \subset B_{t + c \epsilon}
\, .
\]
and since $\chi_\epsilon \geq 0$, again by (\ref{adm1}) : 
\[
\int\limits_{B_{t + c \epsilon}} \, \phi^y_\epsilon ( g^{-1} h \Gamma (q))
\, d \, \mbox{vol} ( g ) \, \geq \, 
\sum\limits_{\gamma y \in B_t \cap \Gamma(q) y} \: \int\limits_{B_{t + c
\epsilon}} \,
\chi_\epsilon ( g^{-1} h \gamma y ) \, d \, \mbox{vol} ( g ) \, \geq \,
| B_t \cap \Gamma(q) y | \, . 
\]
\end{proof}

\medskip
\noindent
\underline{\sf 3.3 \ Uniform error estimates for congruence groups}

We now complete the proof of Theorem 3.2 (compare [GN1, \S 6.6]).

For the lattices $\Gamma ( q )$ the
action of the operators $\pi_{G \slash \Gamma(q)} ( \beta_t)$ on $L^2_0 ( G
\slash \Gamma(q))$ satisfies the spectral estimate stated in Theorem
3.3, uniformly in $q$.  It follows that for the probability spaces $( G \slash \Gamma(q) ,
\, m_{G \slash \Gamma(q)})$ we have for all $t > 0$ and every $\theta^\prime < \theta$
\[
\Big\| \pi_{G \slash \Gamma(q)} ( \beta_t ) \, \phi^y_\epsilon \, - \,
\int\limits_{G \slash \Gamma(q)} \, \phi^y_\epsilon \, dm_{G \slash \Gamma(q)}
\Big\|_{L^2(m_{G/\Gamma(q)})} \, \leq \, C_{\theta^\prime} e^{- \theta^\prime t} \, \parallel
\phi^y_\epsilon \parallel_{L^2(m_{G/\Gamma(q)})}\,\,.
\]
Therefore for all $\delta > 0$, all $t > 0$ and $\epsilon < \epsilon_0^\prime$ 
\[
m_{G \slash \Gamma(q)} 
 \left\{
h \Gamma(q); \Big| \pi_{G \slash \Gamma(q)} ( \beta_t ) \, \phi^y_\epsilon ( h
\Gamma(q) ) \, - \,
\frac{1}{[\Gamma :\Gamma (q)])} 
\Big| \, > \, \delta \right\}  \, \leq \, C^2_{\theta^\prime}
\, \delta^{-2} \, e^{- 2 \theta^\prime t} \, \parallel \phi^y_\epsilon
\parallel^2_{L^2 ( m_{G \slash \Gamma(q) )}} \, .
\]

Clearly, we can fix $\epsilon^{\prime\prime}_0$ such that if $\epsilon <
\epsilon^{\prime\prime}_0$ then the translates ${O}_\epsilon
w$ are disjoint for distinct $w \in \Gamma (1)$.  Then the
supports of the functions $\chi_\epsilon ( h \gamma y )$ for $\gamma
\in \Gamma(q)$ (and a fixed $y \in \Gamma ( 1 ))$ do not intersect, and so
\[
\parallel \phi^y_\epsilon \parallel^2_{L^2 ( m _{G \slash \Gamma(q) })} \, =
\, \int\limits_{G \slash \Gamma(q)} \, \phi^y_\epsilon ( h \Gamma(q) )^2 \, d
\, \mbox{vol}_{G \slash \Gamma(q)} \slash [ \Gamma:\Gamma(q)]  \, =
\]
\[
=  \, \int\limits_{G} \, \chi^2_\epsilon ( g ) \, \mbox{dvol} ( g )
\slash [\Gamma:\Gamma(q)] )=\frac{1}{\mbox{vol}(\cO_\epsilon) [\Gamma:\Gamma(q)]} \, \leq \, \frac{2 \epsilon^{-\dim G}}{[ \Gamma:\Gamma(q)] }\,\,.
\]
We conclude that 
\begin{equation}
m_{G \slash \Gamma(q)} 
\left\{
h \Gamma(q)\,;\, \Big| \pi_{G \slash \Gamma(q)} ( \beta_t) \, \phi^y_\epsilon ( h
\Gamma(q)) \, - \, \frac{1}{[ \Gamma:\Gamma(q)] )} \Big| \, > \, \delta \right\} \, \leq \, \\ \frac{2C^2_{\theta^\prime} \,\delta^{-2} \,
\epsilon^{-\dim G} \, e^{-2 \theta^\prime t}}{[ \Gamma:\Gamma(q)] } \, .
\end{equation}
In particular, the measure of the latter set decays exponentially fast with $t$.
Therefore, it will eventually be strictly smaller than $m_{G \slash
\Gamma(q)} ( {O}_\epsilon \Gamma)$, and for $\epsilon <
\epsilon^{\prime\prime}_0$, we clearly have $m_{G \slash \Gamma(q)} (
{O}_\epsilon \Gamma(q) ) \, = \, \mbox{vol} ( {O}_\epsilon
) \slash [ \Gamma:\Gamma(q)] $.

For any $t$ such that the measure in (3.5) is sufficiently small,
clearly 
\begin{equation}
{O}_\epsilon \Gamma (q)\cap \left\{
h \Gamma(q)\, ; \, \Big|
\pi_{G \slash \Gamma(q)} ( \beta_t) \, \phi^y_\epsilon ( h \Gamma(q) ) \, - \,
\frac{1}{[ \Gamma :\Gamma(q)] } \Big| \, \leq \, \delta \right\} \, \ne \,
\emptyset
\end{equation}
and thus for any $h_t$ such that $h_t \Gamma(q)$ is in the non-empty
intersection (3.6)
\[
\frac{1}{\mbox{vol} (B_t)} \, \int\limits_{B_t} \, 
\phi^y_\epsilon ( g^{-1} \, h_t \Gamma(q) ) \, 
d \, \mbox{vol} ( g ) \, \leq \, 
\frac{1}{[ \Gamma:\Gamma(q)]  } \, + \, \delta \, .
\]
On the other hand, by Lemma 3.5 for any $\epsilon \leq \epsilon_0$, $t
\geq t_0 + c \epsilon_0$ and $h \in \mathcal{O}_\epsilon$
\begin{equation}
| \Gamma(q) y \cap B_t | \, \leq \,
\int\limits_{B_{t + c \epsilon}} \, \phi^y_\epsilon ( g^{-1} h \Gamma(q) )
\, d \, \mbox{vol} ( g ) \, .
\end{equation}
Combining the foregoing estimates and using (\ref{adm2}),  we conclude that
\[
| \Gamma(q) y \cap B_t | \leq \left(
\frac{1}{[\Gamma :\Gamma(q)]} \, + \, \delta \right) \, \mbox{vol} ( B_{t + c
\epsilon} ) \, \leq \,
\left( 
\frac{1}{[\Gamma:\Gamma(q)] } \, + \, \delta \right) \,
( 1 + c \epsilon) \, \mbox{vol} ( B_t) \, .
\]
This estimate holds as soon as (3.6) holds, and so certainly when  
\[
\frac{2 C^2_{\theta^\prime}}{[ \Gamma:\Gamma] } \, \delta^{-2} \, \epsilon^{-\dim G} \, e^{- 2
\theta^\prime t} \, \leq \, \frac{1}{2} \, . \,
\frac{\frac{1}{2} \epsilon^{\dim G}}{[ \Gamma:\Gamma(q)]} \, \leq \, \frac{1}{2} \,
m_{G \slash \Gamma} \, ( \mathcal{O}_\epsilon \Gamma )
\, .
\]

Thus we seek to determine the parameters so that $8C^2_{\theta^\prime} \delta^{-2} \,
\mbox{exp} ( - 2 \theta^\prime t ) \, = \, \epsilon^{2\dim G}$.  In order to 
balance the two significant parts of the error term, let us take $c \epsilon =
\delta$, and then 
\[
\delta \, = \, C_{\theta^\prime}^\prime\, e^{- 2 \theta^\prime t \slash (2 \dim G  + 2)}
\]
and so as soon as $\delta < 1$, we have,  using also  that $[\Gamma : \Gamma(q)]\geq 1$ 
\[
\frac{| \Gamma(q) y \cap B_t|}{ \mbox{vol} ( B_t)} \, \leq \,
\left(
\frac{1}{[ \Gamma:\Gamma(q) ]} \, + \, \delta \right) \,(1 \, + \, c \epsilon )
\, \leq \, \frac{1}{[ \Gamma:\Gamma(q)] } \, + \,  \delta +c\epsilon+\delta c \epsilon
\]
\[ 
\leq \, \frac{1}{[\Gamma: \Gamma(q) } \, + \, 3 C^\prime_{\theta^\prime} \, e^{- \theta^\prime
t \slash ( \dim G + 1 )} \, . 
\]

Note that both the estimate (3.4) as well as the comparison argument in
Lemma 3.5, give a lower bound in addition to the foregoing upper bound.
Thus the same arguments can be repeated to yield also a lower bound for
the uniform lattice points count.  This concludes the proof of Theorem 3.2.
\hfill $\square$

\noindent {\bf Remark.} When the admissible family of sets $B_t$ consists of bi-$K$-invariant sets, namely sets that are invariant under left and right multiplication by a maximal compact subgroup $K$ of $G$, two improvements are possible in the previous result.
\begin{namelist}{xxxxx}
\item[(i)]  First, the parameter $\theta$ which controls the exponential decay of the operator norm $\parallel \pi^0_{G/\Gamma(q)}(\beta_t)\parallel$ depends 
only on the spherical spectrum and can be estimated directly by the spectral theory of spherical functions. The resulting estimate is $\theta=a/p_K $, where $p_K(G,\Gamma)$ is  the $L^p(G)$-integrability parameter associated with  the spherical functions in $\pi^0_{G/\Gamma}$. 
This estimates eliminates the lack of resolution that can be caused by the tensor power argument,      
which gives $\theta =a/2n_e$, $n_e$ the least even integer $\ge p(G,\Gamma)/2$.  
\item[(ii)] Second, when $B_t$ are bi-$K$-invariant, the arguments in the proof  of Theorem 3.2 can be applied 
in the obvous manner to the symmetric space $G/K$ whose dimension is $\dim G \,-\, \dim K$ , so that the exponent in the resulting error estimate is $\frac{a}{p_K (1+\dim G/K)}$.
  \end{namelist}
  
\medskip
\noindent
\underline{\large \sf Section 4. \ Multiplicativity and Sieve Dimension}

As before, let $\mathbb{G} \subset GL ( n,\mathbb{C} )$ be an algebraically connected semisimple algebraic
matrix group defined over $\mathbb{Q}$ which we now assume is also
simply connected.  Denote by $\mathbb{G} ( R )$ the points of $G$ with
coefficients in a ring $R$, and set $G=\mathbb{G}(\mathbb{R})$.  Fix $v_0 \in \mathbb{Z}^n$ and let $V = \mathbb{G}. v_0$ be the corresponding orbit which we assume is Zariski closed
in affine $n$-space $A^n$.  We assume further that the stabilizer of $v_0$
in $\mathbb{G}$ is trivial.  Thus $V$ is a principal homogeneous space for $\mathbb{G}$
and it is defined over $\mathbb{Q}$.  Since $\mathbb{G}$ is connected it follows
that $V$ is an (absolutely) irreducible affine variety defined over
$\mathbb{Q}$ and is of dimension equal to $\dim \mathbb{G}$.  The ring of
$\mathbb{G}$-invariants for the action of $\mathbb{G}$ on $n$-dimensional space separates
the closed $\mathbb{G}$-orbits (see [B-HC] and we may choose generators $h_1 ,
h_2 , \ldots , h_\nu$ in $\mathbb{Q} [ x_1,\ldots, x_n]$ of this ring so that $V$ is
given by
\setcounter{section}{4}
\setcounter{equation}{0}
\begin{equation}
V: \, h_j ( x ) \, = \, \lambda_j \ \mbox{for} \ j \, = \, 1 , \ldots
\nu \ \mbox{and} \ \lambda_j \in \mathbb{Q} \, .
\end{equation} 

Let $V ( \mathbb{Z} )$, $V ( \mathbb{Q})$ denote the points of $V$ with
coordinates in $\mathbb{Z} , \mathbb{Q}$, etc.

\medskip
\noindent
\underline{\sf 4.1 \ Congruential analysis:}

Let $\Gamma =  G ( \mathbb{Z} )
\subset GL_n ( \mathbb{Z} )$ and
$\mathcal{O}  =  \Gamma . v_0$ 
be the corresponding orbit in
$\mathbb{Z}^n$. Since $\mbox{Zcl} ( \Gamma ) \, = \, G$ (see [Bo]),
$\mbox{Zcl} ( \mathcal{O} ) = V$.  For 
$d \geq 1$ an integer let
$\mathcal{O}_d$ be the subset of 
$( \mathbb{Z} \slash d \mathbb{Z} )^n$
which is obtained by reducing 
$\mathcal{O} $ modulo $d$.  Similarly let
$\Gamma_d$ be the reduction of $\Gamma$ in 
$GL_n ( \mathbb{Z} \slash d \mathbb{Z} )$, and so $\mathcal{O}_d = \Gamma_d \, . v_0  \, ( \mbox{mod } d )$. 

 For $g
\in \mathbb{Z} [ x_1,\ldots, x_n ]$ let
\begin{equation}
\mathcal{O}^g_d \, = \, \left\{
x \in \mathcal{O}_d : \, g ( x ) \, \equiv \, 0 ( d ) \right\} \, .
\end{equation}
According to the strong approximation theorem (see [P-R], recall that
we are assuming that $G ( \mathbb{R})$ has no compact factors) the
diagonal embedding 
\begin{equation}
\Gamma \longrightarrow \,\prod\limits_{p} \, G ( \mathbb{Z}_p )
\end{equation}
is dense, where $\mathbb{Z}_p$ are the $p$-adic integers.  Hence if $(d_1
, d_2) = 1$ then 
\begin{equation}
\Gamma_d \, = \, \Gamma_{ d_1} \, \times \, \Gamma_{ d_2}
\end{equation}
as a subgroup of $GL_n ( \mathbb{Z} \slash d \mathbb{Z} ) \, \cong
\, GL_n ( \mathbb{Z} \slash d_1 \mathbb{Z}) \, \times \, GL_n (
\mathbb{Z} \slash d_2 \mathbb{Z})$.

It follows that in $( \mathbb{Z} \slash d \mathbb{Z})^n \simeq (
\mathbb{Z} \slash d_1 \mathbb{Z})^n \, \times \, ( \mathbb{Z} \slash d_2
\mathbb{Z})^n$
\begin{equation}
\mathcal{O}_d \, = \, \mathcal{O}_{d_1} \times \, \mathcal{O}_{ d_2}
\end{equation}
and
\begin{equation}
\mathcal{O}^g_d \, = \, \mathcal{O}^g_{d_1} \, \times \,
\mathcal{O}^g_{d_2} \, .
\end{equation}

Now let $ f \in \mathbb{Q} [ x_1,\ldots,x_n ]$ with $f = g \slash N$, $N \geq 1$
where $g \in \mathbb{Z} [ x_1,\ldots,x_n ]$ with $g c d ( g ( \mathcal{O})) = N$.
Note that $f ( \mathcal{O}) \subset \mathbb{Z}$.  For $d \geq 1$ let 
\begin{equation}
\rho_f ( d ) \, = \, \frac{d | \mathcal{O}^g_{d N} |}{| \mathcal{O}_{d
N}|} \, .
\end{equation}
\setcounter{prop}{4}
\setcounter{prop}{0}
\begin{prop}
$\rho_f ( d )$ is multiplicative in $d$ and for $p$ prime $0 \leq \rho_f
( p ) < p$.
\end{prop}
\begin{proof}
Let $d = d_1 d_2$ with $(d_1 , d_2) = 1$ and write $N = N_1 N_2$ with
$(N_1 , d_2) = 1$, $(N_2 , d_1) = 1$ and $(N_1 , N_2 ) = 1$. Clearly :
\[
| \mathcal{O}_{d_1 d_2 N_1 N_2} | \, = \,
 \frac{|\mathcal{O}_{d_1 N_1 N_2}| \, | \mathcal{O}_{d_2 N_2 N_1}
|}{| \mathcal{O}_N |}\,=\,
 \frac{| \mathcal{O}_{d_1 N} | \, | \mathcal{O}_{d_2N} | }{
 | \mathcal{O}_N |}\, \, . 
\]
Furthermore :
\[
\begin{array}{lll}
| \mathcal{O}^g_{d_1 d_2 N_1 N_2} | & = & | \mathcal{O}^g_{d_1N_1} | \,
. \, | \mathcal{O}^g_{d_2 N_2 } | \\
\\
& = & \frac{| \mathcal{O}^g_{d_1 N_1 N_2}|}{| \mathcal{O}^g_{N_2} |} \,
. \, 
\frac{| \mathcal{O}^g_{d_2N_2 N_1}|}{| \mathcal{O}^g_{N_1} |} \\
\\
& = & \frac{| \mathcal{O}^g_{d_1 N} | \, | \mathcal{O}^g_{d_2N} | }{
 | \mathcal{O}^g_N |}\, \, .
\end{array}
\]
and hence
\[
\begin{array}{lll}
\rho_f ( d_1 d_2) & = & 
\frac{| \mathcal{O}^g_{d_1d_2 N} |}
{| \mathcal{O}_{d_1 d_2 N |}} \, = \,
\frac{| \mathcal{O}^g_{d_1 N} | \, | \mathcal{O}^g_{d_2 N} |}
{| \mathcal{O}^g_N |} \, \frac{| \mathcal{O}_N |}{| \mathcal{O}_{d_1 N}
| \, | \mathcal{O}_{d_2 N} |}  \\
\\
& = & \rho_f ( d_1 ) \, \rho_f ( d_2 )  
\end{array}
\]
since  \hspace{1.25in}  $| \mathcal{O}^g_N | \,  =  \, | \mathcal{O}_N |$.

For $d = p$ there is $x \in \mathcal{O}$ s.t.
\[
\frac{g ( x )}{N} \, \not\equiv \, 0 ( p ) \ \mbox {since}  \ g c d ( g (
\mathcal{O})) \, = \, N \, . \]
Hence $x \notin \mathcal{O}^g_{d N} \Longrightarrow \, \rho_f ( p ) \, <
\, p$. 
\end{proof}

Factoring $f \in \mathbb{Q} [V]$ into $t=t(f)$ irreducibles we get $f = f_1 f_2 \ldots
f_t$, where we are assuming further that 
each $f_j$ is irreducible in $\mathbb{C} [
V]$ and that the $f_j$'s are distinct.  According to E. Noether's
Theorem [No] there is a finite set of primes $S$ such that for $p
\notin S$, $V$ is absolutely irreducible over $\mathbb{Z} \slash p
\mathbb{Z} = \mathbb{F}_p$.  By increasing the set $S$ if necessary, we
can assume that for $p \notin S$ the equations defining $G$ also yield
an absolutely irreducible variety over $\mathbb{F}_p$.  According to
Lang's Theorem [La] we have that
\begin{equation}
V ( \mathbb{Z} \slash p \mathbb{Z} ) = V ( \mathbb{F}_p ) = \mathbb{G} (
\mathbb{F}_p ) \, .\,  v \,.
\end{equation}
By strong approximation we have (again increasing $S$ if needed) that
for $p \notin S$
\begin{equation}
\mathbb{G} ( \mathbb{Z} \slash p \mathbb{Z}) \, = \, \Gamma_p \, .
\end{equation}
and hence that
\begin{equation}
\mathcal{O}_p \, = \, V ( \mathbb{F}_p ) \hspace{.25in} 
\mbox{for} \ \ p \notin S \, .
\end{equation}
Also when $p$ does not divide  $ N$ we have that 
\[
\mathcal{O}^f_p \, = \, \{ x \in \mathcal{O}_p : \, f ( x ) \, \equiv \,
0 \, ( \mbox{mod}p) \}
\]
is well defined and
\begin{equation}
\frac{| \mathcal{O}^f_p |}{ | {\displaystyle \mathcal{O}_p |}} \, = \, \frac{|
\mathcal{O}^g_{p N}|}{ |{\displaystyle \mathcal{O}_{pN} |}} \, .
\end{equation}

Finally for $1 \leq j \leq t(f)$ let 
\begin{equation}
W_j \, = \, \{ x \in V: \, f_j ( x ) \, = \, 0 \} \, .
\end{equation}

$W_j$ is an (absolutely irreducible) affine variety defined over
$\mathbb{Q}$ of dimension $\dim V - 1$.  Hence again by Noether's
Theorem $W_j$ is absolutely irreducible over $\mathbb{F}_p$ for $p$
outside $S^\prime$ say.  For such $p$ we apply a weak form of the Weil
conjectures to $W_j$ (see [L-W] or [Sch] for an elementary
treatment) to conclude that
\begin{equation}
| W_j ( \mathbb{F}_p ) | \, = \, p^{\dim V - 1} \, + \, {O}
\left(
p^{\dim V - \frac{3}{2}} \right) \, .
\end{equation} 
where the implied constant depends on $j$ only.

Furthermore, since the $f_j$'s are distinct irreducibles in
$\mathbb{C}
[ V ]$, we have for $i\neq j$  
\begin{equation}
\dim ( W_i \cap W_j ) \, \leq \, \dim V - 2 \, .
\end{equation}
Hence for $p \notin S \cup S^\prime$ we have
\begin{equation}
\begin{array}{lll}
| \mathcal{O}^f_p ) | & = & \sum\limits_{j \, = \, 1}^{t(f)}
\, | W_j ( \mathbb{F}_p ) | \, + \, O \left( p^{\dim V-2} \right)
\\
\\
& = & t(f) p^{\dim V-1} \, + \, O \left( p^{\dim V - \frac{3}{2}} \right)
\end{array}
\end{equation}
and similarly
\begin{equation}
| \mathcal{O}_p ) \, | \, = \, | \, V ( \mathbb{F}_p ) | \,
= \, p^{\dim V} \, + \, O \left( p^{\dim V - \frac{1}{2}} \right) \, .
\end{equation}
Combining the above, we have that for $ p \notin S \cup S^\prime$
\begin{equation}
\frac{| \mathcal{O}^f_p |}{| \mathcal{O}_p|} \, = \, \frac{t(f)}{p} \, + \,
O \left( p^{- \frac{3}{2}} \right)
\end{equation}
and hence that 
\begin{equation}
| \rho_f ( p ) \, - t(f) \, | \, \leq \, C p^{ - \frac{1}{2}}
\end{equation}
where $C$ depends only on $\mathcal{O}$ and $f$.

\medskip
\noindent
\underline{\sf 4.2 \ Applying the sieve}

As in (2.9) we define the sequence $a_k ( T )$, $k \geq 0$ by
\begin{equation}
a_k ( T ) \, = \; \sum\limits_{\gamma\in \Gamma\,;\,\parallel \gamma \parallel \, \leq \, T
\atop{| f ( \gamma v ) | \, = \, k}} 1
\, .
\end{equation}

The sums on progressions are then, 
for $d \geq 1$ square free
\begin{equation}
\sum\limits_{k \, \equiv \, 0 ( d )} \, a_k ( T ) \, = \,
\sum\limits_{\gamma \in \Gamma\,;\,\parallel \gamma \parallel \, \leq \, T\atop{f ( \gamma v ) \, \equiv \, 0 ( d )}} \!\!\! 1
\end{equation}
\begin{equation}
= \; \sum\limits_{\delta \, \in \Gamma \slash \Gamma ( dN  )\atop{g (
\delta v ) \, \equiv \, 0 ( dN  )}}  \; 
\sum\limits_{\gamma \, \in \Gamma ( dN  )\atop{\parallel \delta \gamma
\parallel \, \leq \, T}}1
\end{equation}
where $\Gamma ( q )$ is the congruence subgroup of $\Gamma$ of level
$q$ and $f=g/N$ as in \S 4.1.

According to Theorem 3.2, (4.21) becomes 
\[
= \; \sum\limits_{\delta \, \in \Gamma \slash \Gamma ( dN  )\atop{g (
\delta v ) \, \equiv \, 0 ( dN  )}} \
\left( 
\frac{ \mbox{vol} \{ \parallel w \parallel \, \leq \, T \} }
{[\Gamma: \Gamma ( dN  ) ]} \, + \, O_\epsilon \,
\left(
T^{a - \frac{\theta}{1 + \dim G} \, + \, \epsilon} 
\right) \right)
\]
\[
\,=\,X \cdot \sum\limits_{\delta \, \in \, \Gamma \slash \Gamma ( dN)\atop{g (
\delta v ) \, \equiv \, 0 ( dN  )}} \;
\frac{1}{[ \Gamma: \Gamma ( dN  ) ]} \, + \, O_\epsilon \left(
| \mathcal{O}^g_{dN} | \, T^{a - \frac{\theta}{1 + \dim G} \, + \,
\epsilon} \right)
\]
where
\begin{equation}
X \, = \, \sum\limits_{k\in \mathbb{N}} \, a_k ( T ) \, .
\end{equation}
Now
\[
\mathcal{O}_{dN} \, = \, \Gamma_{dN }\,.\, v (\mbox{mod }dN )
\]
and hence
\begin{equation}
| \mathcal{O}_{dN} | \, = \, \frac{| \Gamma_{dN} |}{|H_{dN}|}
\end{equation}
where $H_{dN}$ is the stabilizer of $v$ in $\Gamma_{dN}$.  
Also
$\Gamma \slash {\Gamma ( dN)} \, \cong \, \Gamma_{dN}$ and so 
\begin{equation}
| \{ \delta \, \in \, \Gamma_{dN} : \, g ( \delta v ) \, \equiv \, 0 (
dN ) \} | \, = \, | \mathcal{O}^g_{dN} | \, | H_{dN} | \, . 
\end{equation}
Thus (4.20) becomes
\begin{equation}
\sum\limits_{k \, \equiv \, 0 ( d )} \, a_k ( T ) \, = \, 
\frac{X | \mathcal{O}^g_{dN} | \, . \, | H_{dN} |}{| \Gamma_{dN} |}
\, + \, O_\epsilon \left(
d^{\dim G} \, T^{a - \frac{\theta}{1 + \dim G} \, + \epsilon}
\right)
\end{equation}
where we have used $| \mathcal{O}^g_{dN} | \ll d^{\dim G}$, though for
later note that from (4.15) and (4.6)  
\begin{equation}
| \mathcal{O}^g_{dN} | \ll d^{\dim G - 1} \, .
\end{equation}
Hence from (4.23) and (4.25) we have 
\begin{equation}
\sum\limits_{k \, \equiv \, 0 ( d )} \, a_k ( T ) \, = \,
\frac{\rho_f ( d )}{d} \, X \, + \, R ( \mathcal{A} , d )
\end{equation}
with
\begin{equation}
\rho_f ( d ) \, = \, \frac{d | \mathcal{O}^g_{dN} |}{|
\mathcal{O}_{dN}|}
\end{equation}
and
\begin{equation}
\begin{array}{ll}
| R( \mathcal{A} , d ) | & \mathop{\ll}\limits_{\epsilon} \, 
d^{\dim G} \, T^{a \left(1- \frac{\theta}{a(1 + \dim G)}\right) \, + \, \epsilon} \\
\\
&\mathop{\ll}\limits_{\epsilon} \, 
d^{\dim G} \, X^{1 - \frac{\theta}{a (
1 + \dim G )} \, + \, \epsilon} \\
\\
& \mathop{\ll}\limits_{\epsilon} \,  d^{\dim G} \, X^{1 - \frac{1}{2n_e (1 + \dim G )} \, + \, \epsilon}
\end{array}
\end{equation}
according to Theorem 3.3, since $\theta=a/2n_e$.

By Proposition 4.1, (4.27) establishes axiom $(A_0)$ with $B$ the empty
set.
As for the level distribution $(A_1)$ we have from (4.29) that
\begin{equation}
\begin{array}{ll}
\sum\limits_{d   \leq \, D} \, | R( \mathcal{A} , d )| & 
\mathop{\ll}\limits_{\epsilon} \, D^{1 + \dim G} \, X^{ 1 - \frac{1}{2n_e
( 1 + \dim G)} \, + \, \epsilon} \\
\\
& = \, O ( X^{1- \zeta})
\end{array}
\end{equation}
as long as
\begin{equation}
D \, \leq \, X^\tau \ \mbox{with} \ \tau \, < \, \frac{1}{2n_e ( 1 + \dim
G )^2} \, .
\end{equation}
Finally axiom $(A_2)$ follows with a suitable $C_2 = C_2 ( \mathcal{O},
f )$ from (4.18).  We apply the combinatorial sieve in the form (2.8) to
conclude that for 
\begin{equation}
z = X^\alpha \ \mbox{ with} \ \alpha = \frac{1}{9 t(f) ( 1 + \dim G )^2 \cdot 2n_e}
\end{equation}
and for
\[
P \, = \, \prod\limits_{p \le z} \, p \, ,
\]
with $X$ large enough,  
\begin{equation}
\frac{X}{( \log X )^{t(f)}} \, \ll \, S ( \mathcal{A} , P ) \, \ll \, \frac{X}{( \log X
)^{t(f)}}
\end{equation}
where the implied constants depend only on $f$ and the orbit $\mathcal{O}$.

\medskip
\noindent
\underline{\sf 4.3 \ Completion of proofs of Theorems 1.6, 1.7 and
Corollary 1.8}

We begin by establishing the following two Lemmas.

\setcounter{Lem}{1}
\begin{Lem}
Assume that  $ h \in \mathbb{Q} [ x_1 , \ldots , x_n ]$ does not vanish identically when restricted to
 $V = G\, .\,  v$. Then there is $\delta = \delta ( h
) > 0$ such that
\[
| \{ \gamma \in \Gamma: \, 
\parallel \gamma \parallel \, \leq \, T, \, h
( \gamma v ) \, = \, 0 \} | \, \ll \, T^{a - \delta} \, . \]
\end{Lem}
\noindent
{\it Proof.}
We may assume that $h$ is not constant on $V$ and hence there is a
finite extension $E$ of $\mathbb{Q}$ over which $h$ factors into $h =
h_1 h_2 \ldots h_\nu$ where each $h_j$ is absolutely irreducible in
$E[V]$.  Hence for $p$ large enough and $\mathcal{P}$ a prime ideal in the ring of integers 
$\mathcal{I}_E$  of $E$ with $\mathcal{P}| ( p
)$ we have 
\[
\# \{ x \in V ( \mathcal{I}_E \slash \mathcal{P} ) : \, h_j ( x ) \, = \, 0 \} \,
\ll \, N ( \mathcal{P})^{\dim V-1} \, . \]
\begin{equation}
\hspace{-3.75in} \mbox{The implied constant depending on} \ h \, .
\end{equation}
Assume further that $p$ splits completely in $E$ so that $\mathcal{I}_E
\slash \mathcal{P} \, \cong \, \mathbb{Z} \slash p \mathbb{Z}$. Then 
\begin{equation}
\# \{ x \in V ( \mathbb{Z} \slash p \mathbb{Z}): \, h ( x ) \, \equiv \,
0 ( p ) \} \, \ll \, p^{\dim V - 1} \,.
\end{equation}
Let $T$ be the large parameter in the Lemma and choose $p$ as above with
$T^\alpha \slash 2 \, \leq \, p \, \leq \, 2 T^\alpha$ for $\alpha > 0$
small and to be chosen momentarily.  Such a $p$ exists by Chebotarev's
density theorem [Che].  With this choice we have that
\begin{equation}
\begin{array}{ll}
| \{ \gamma \in \Gamma: \, \parallel \gamma \parallel \, \leq
\, T \,,\, h ( \gamma v ) \, = \, 0 \} | \\
\\
\,  \leq \, | \{ \gamma \in \Gamma : \, \parallel \gamma \parallel \, \leq
\, T \, , \, h ( \gamma v ) \, \equiv \, 0 ( p ) \} | \\
\\
\,  = \, \frac{| \mathcal{O}^h_p |}{| \mathcal{O}_p|} \, X \, + \, O_\epsilon \left(
T^{a - \frac{\theta}{1 + \dim G} \, + \, \epsilon} \, p^{\dim G }
\right) 
\end{array}
\end{equation}
on using (4.25).

Coupled with (4.35) this gives 
\begin{equation}
\begin{array}{l}
| \{ \gamma \in \Gamma: \, \parallel \gamma \parallel \, \leq \, T \, ,
\, h ( \gamma v ) \, = \, 0 \} | \\
\\
\, \mathop{\ll}\limits_{\epsilon} \, \frac{T^{a + \epsilon}}{p} \, + \,
T^{a - \frac{\theta}{1 + \dim G} + \epsilon} \, p^{\dim G} \\
\\
\, \ll \, T^{a - \alpha} \, , \hspace{.5em}  
\mbox{where we choose} \ \alpha \, = \, \frac{\theta}{(1 + \dim G)^2} \,
.
\end{array}
\end{equation}
\begin{Lem}
Let $f\, = \, f_1 f_2 \ldots f_t$,  with $f_j \in \mathbb{Z}[ x_1,\ldots, x_n ]$
irreducible as in Theorem 1.6, $1 \le j \le t(f)$.  Then there is $\delta _1 > 0$ such that for
any $m \in \mathbb{Z}$ and any $1 \leq j \leq t(f)$,
\[
| \{ \gamma \in \Gamma: \, \parallel \gamma \parallel \, \leq \, T \, ,
\, f_j ( \gamma v ) \, = \, m \} | \, \ll \, T^{a - \delta_1}
\]
the implied constant depending only on $f$ and $\Gamma$.
\end{Lem}
\begin{proof}
By assumption $f_j$ is not constant when restricted to $V$.  Hence by
Lemma 4.2 with $g = f_j - m$ we get Lemma 4.3 but potentially the
implied constant depends on $m$.  The only place where this dependence
may enter is in 4.34 and for $m$ outside a finite set $f_j - m$ is
irreducible over $\bar{\mathbb{Q}}$.  The equations defining $f_j - m =
0$ and $V$ will be irreducible over $\bar{\mathbb{F}}_p$ for $p$ outside
a fixed finite set and they are of a fixed degree in the variables $(x_1 ,
\ldots , x_n)$.  Hence by a Lemma of Lang-Weil [L-W] the upper bound in
(4.34) with $h = f_j - m$ is uniform in $m$.
\end{proof}

\noindent
\underline{\it Proof of Theorem 1.6:}

We choose $\epsilon_1 > 0$ small (but fixed) so that firstly $\epsilon_1
< \delta_1$ where $\delta_1$ is determined by Lemma 4.3.  For $1 \leq j \leq t(f)$ and
$T$ large we have from Lemma 4.3 that 
\begin{equation}
\# \{ \gamma \in \Gamma : \, \parallel \gamma \parallel \, \leq \, T \,
, \, | f_j ( \gamma v ) | \, \leq \, T^{\epsilon_1} \} \, \ll \, T^{a -
\delta_1 \, + \, \epsilon_1} \, .
\end{equation}
\hfill $\square$

Now
\[
\{ 
\gamma \in \Gamma: \, \parallel \gamma \parallel \, \leq \, T \, , \,
f_j ( \gamma v ) \hspace{.9em}  \mbox{is prime} \} \ 
\subset \, 
\mathop{\cup}\limits_{j \, = \, 1}^{t(f)} \, 
\{ 
\gamma \in
\Gamma: \, \parallel \gamma \, \parallel \, \leq \, T \, , \, | f_j (
\gamma v ) | \, \leq \, T^{\epsilon_1} 
\} \,   
{\cup} \] 
\begin{equation}
\left\{ 
\gamma \in \Gamma : \, \parallel \gamma \parallel
\, \leq \, T \, , \, | f_j ( \gamma v ) | \, \geq \, T^{\epsilon_1} \, 
\right.
\\
\\
\left.
\mbox{for each} \ j \\,,\, \mbox{and} \ f_j ( \gamma v ) \ 
\mbox{is prime}
\right\} \,. 
\end{equation}

 By (4.38) the cardinality of the union of the first  $t(f)$ sets above is at most $ O ( T^{a - \delta_1 + \epsilon_1})$. The last set on the right hand side of (4.39) is contained in 
\[
\{ \gamma \in \Gamma: \, \parallel \gamma \parallel \, \leq \, T \, , \,
( f ( \gamma v ) \, , \, P_z ) \, = \, 1 \}
\]
where
\begin{equation}
P_z \, = \prod_{p \, \leq \, z }  \,  \, z \, = \,
T^{\epsilon_1} \, .
\end{equation}
The cardinality of the last set is $S ( \mathcal{A} , P_z )$ and if $\epsilon_1 <
\alpha$ where $\alpha$ is the level distribution in (4.32) then we may
apply (4.33) to conclude that 
\[
\begin{array}{l}
\{ 
\gamma \in \Gamma: \, \parallel \gamma \parallel \, \leq \, T \, , \,
f_j ( \gamma v ) \ \mbox{is prime} 
\} \\
\\
\hspace{.10in } \ll \, T^{a - \delta_1 \, + \, \epsilon_1} \, + \, 
 \frac{X}{( \log X )^{t(f)}} \, \ll \, \frac{X}{( \log X )^{t(f)}} \, .
\end{array}
\]
This completes the proof of Theorem 1.6.
\hfill $\square$

\noindent
\underline{\it Proof of Theorem 1.7:}

Taking $\alpha$ as in (4.32) and $z = X^\alpha$, we have that
\begin{equation}
\sum\limits_{\parallel \gamma \parallel \, \leq \, T\atop{(f ( \gamma v )
\, , \, P_z ) \, = \, 1}} 1 \: \gg \, \frac{X}{( \log X )^{t(f)}}
\end{equation}
where $P_z \, = \, \prod\limits_{p \, \leq \, z} \, p$. 

\noindent
Now any point $\gamma v \in \mathcal{O}$ which occurs in the sum in
(4.41) has $| \gamma v | \, \ll \, T$ (where $| \ |$ is the usual
Euclidean norm on $\mathbb{R}^n$) and hence $| f (\gamma v ) | \, \ll \,
T^{\text{deg }f}$.  On the other hand for such a point 
$\gamma v \in \mathcal{O}$ in the sum, $f ( \gamma v )$ has all its
prime factors at least $z \gg T^{a \alpha}=X^\alpha$.  It follows that for such a point 
$\gamma v$, $f ( \gamma v )$ has at most
\[
r \, = \, \frac{\mbox{deg }f}{a\alpha } \, = \, \frac{9 t(f) ( 1 \, + \, \dim
G)^2\cdot 2n_e(\Gamma) \mbox{deg }f}{a}
\]
prime factors. \hfill $\square$

\noindent
\underline{\it Proof of Corollary 1.8:}

Suppose by way of contradiction that the points $\gamma v$ that we
produced in the previous paragraph are not Zariski dense in $V$.  Since
$V$ is connected it follows that there is a $h \in \mathbb{Q} [ x_1 ,
\ldots , x_n ]$ which does not vanish identically on $V$ and such that all our points
lie in $V \cap \{ x: h ( x) = 0 \}$.  But by Lemma 4.2 the total number
of points in this intersection with $\parallel \gamma \parallel \, \leq
\, T$ is $O ( T^{a -  \delta})$ with $\delta=\delta(h) > 0$.  This contradicts the
lower bound of $c X \slash ( \log X )^{t(f)}$ (with $c > 0$ fixed) for the
number of points with at most $r$ prime factors that was produced in
Theorem 1.7. \hfill $\square$

Finally we apply the Theorems 1.6 and 1.7  to the case of $n \times n$ integral
matrices of determinant $m$ as in the introduction.  Let $G = SL_n$,
$\Gamma = SL_n ( \mathbb{Z})$ and $v \in {\rm Mat}_{n } (
\mathbb{Z})$ with $\det v \, =\, m \, \ne \, 0$.  We identify ${\rm
Mat}_{n }$ with affine  space $A^M$ ($M = n^2$), with $G$ acting by $x
\rightarrow g x$.  We have $V_{m,n} = G.v$ and $V_{ m,n} (
\mathbb{Z})$ consists of a finite number of $\Gamma$-orbits.  Note that
if $| \ |$ is a norm on ${\rm Mat}_{n \times n} ( \mathbb{R})$ then for
a fixed invertible 
$v$ as above $| g v |$ defines a (vector space) norm on $M_M(\mathbb{R})$ and we may apply Theorems 1.6 and 1.7 to this
setting, namely to the individual orbits $\mathcal{O} = \Gamma . v$ with
$v$ fixed as above.  In this case $\dim G + 1 = n^2$ and $a = n^2 - n$ for any choice of norm as above 
[D-R-S] [G-W][Ma]. Furthermore $p(\Gamma)=2(n-1)$ (see [D-R-S]) and so  $n_e = (n - 1)$ if $n$ is odd, and $n_e=n$ if $n\ge 4$ is even. 
For $n = 2$ we can take $n_e =2$ by [K-S],  see (6.17) . 
Thus Theorem 1.7 yields that for $n$ odd $r$ need only satisfy 
\begin{equation}
r  \, >
\frac{9 . t(f) . n^4. 2( n - 1 )  \mbox{deg }(f)}
{n ( n - 1)} \, = \, 18 . t(f) .n^3. \mbox{deg }(f)
\end{equation} 
and for $n$ even the previous expression is multiplied by $n/(n-1)$. 

Thus for $\mathcal{O} = \Gamma .  v$ and $f$ primitive on $\mathcal{O}$
\begin{equation}
\begin{array}{l}
| \{ x \in \mathcal{O}: \, | x | \, \leq \, T \, , \, f ( x ) \
\mbox{has a most} \,  \ r \ \mbox{prime factors} \} | \\
\\
\hspace{1.25in} \gg \, {T^{n^2 - n}}\Big/{( \log T )^{t(f)}} \, .
\end{array}
\end{equation}
This proves the $\Gamma$-orbit version of Theorem 1.2 and Corollary 1.3.

Now for $m \ne 0$ fixed $V_{ m,n} ( \mathbb{Z})$ consists of a finite
number of such $\Gamma$-orbits, and Theorem 1.1 follows from the $\Gamma$-orbit version.  In order to establish
Theorem 1.2 and Corollary 1.3 for $V_{ m,n} ( \mathbb{Z})$ we need to
show that if $f$ is $V_{m,n} ( \mathbb{Z})$ weakly primitive then there is a
$v \in V_{m,n} ( \mathbb{Z})$ such that if $\mathcal{O} = \Gamma . v$
then $f$ is $\mathcal{O}$-weakly-primitive.

As in the theory of Hecke correspondences on $n$-dimensional lattices 
(see [Te]) we decompose $V_{m,n}$ into $\Gamma$-orbits 
\begin{equation}
V_{m,n} ( \mathbb{Z} ) \, = \, 
\mathop{\coprod}\limits_{j \, = \, 1}^{k( m )} \,
\mathcal{O}^{(j)}
\end{equation} 
with $\mathcal{O}^{(j)} \, = \, \Gamma v_j$, $v_j$ in $V_{m,n}
( \mathbb{Z})$.

Denote by $W$ the union of the $k( m )$ global $\Gamma$-orbits and for $d \geq 1$
let $\mathcal{O}^{(j)}_d$ denote the reduction of $\mathcal{O}^{(j)}
\mod d$ which defines a point in the orbit space $SL_n ( \mathbb{Z} \slash d
\mathbb{Z}) \backslash \mbox{Mat}_n ( \mathbb{Z} \slash d \mathbb{Z} ))$ 
where, $\mbox{Mat}_n$ is the space of $n \times n$ matrices.  Let $W_d$ denote the
reduction of $W$ into this space.  Note that for $d = p$ with $p$ a prime that does not divide $m$, 
$W_p$ consists of a single point, that is to say the orbits
$\mathcal{O}^{(j)}$ all reduce to the same $SL_n ( \mathbb{Z} \slash d
\mathbb{Z})$-orbit modulo $p$.  The key property that we need for these
reductions is that if $(d_1 , d_2) = 1$ then the diagonal embedding
\begin{equation}
\begin{array}{l}
W \longrightarrow ( SL_n ( \mathbb{Z} \slash d_1 \mathbb{Z} ) \backslash \, \mbox{Mat}_n
( \mathbb{Z} \slash d_1 \mathbb{Z} ) ) \, \times \, ( SL_n ( \mathbb{Z} \slash
d_2 \mathbb{Z} ) \backslash \mbox{Mat}_n ( \mathbb{Z} \slash d_2 \mathbb{Z} ))
\end{array}
\end{equation} 
is onto $ W_{d_1} \, \times \, W_{d_2} $. 

With (4.45), the weak primitivity property that we need is established as
follows.  Let $f$ be weakly primitive on $V_{m,n} ( \mathbb{Z})$ and for
simplicity of notation assume that $N = 1$ in Section 1.A and that  $m$
is square free.  So $f \in \mathbb{Z} [ x_{ij} ]$ and for each prime $p
\geq 2$ there is an $x \in V_{m,n} ( \mathbb{Z})$
such that $f ( x ) \ne 0 (\mbox{mod } p )$.  We claim that there is a $J \in
\{ 1 , \ldots , k (  m ) \}$ such that $f$ is weakly primitive for
$\mathcal{O}^{(J)}$.  That is for every prime $p \geq 2$ there is $x \in
\mathcal{O}^{(J)}$ with $f ( x ) \ne 0 ( p )$.

Call a prime $p \geq 2$ good for $\mathcal{O}^{(j)}$ if such an $x$
exists for $p$.  This property is determined locally at $p$.  That is by
strong approximation for $SL_n$, $p$ is good for $\mathcal{O}^{(j)}$ {\sf
iff} the local orbit $SL_n ( \mathbb{Z} \slash p \mathbb{Z} ) v_j$
in $\mbox{Mat}_n ( \mathbb{Z} \slash p \mathbb{Z} )$ contains an $x$ such that $f
( x ) \not\equiv 0 ( p )$.  So the condition is one on
$\mathcal{O}^{(j)}_p$.  Every prime $p$ that does not divide $m$ is good for any
$\mathcal{O}^{(j)}$, $j = 1 , \ldots , k (m)$ since $V_{m,n } (
\mathbb{Z} )$ is good at $p$ and all global orbits reduce to the
same local orbit at such a $p$.  Now write $m = p_1 p_2 \dots p_\ell$, and 
then
\begin{equation}
W \, \longrightarrow \, W_{p_1} \, \times \, W_{p_2} \, \times \dots \,
\times \, W_{p_\ell}
\end{equation}
is onto.

Moreover by our assumption on $f$, for each $p_i , i = 1 , \ldots ,
\ell$, there is $j_{p_i}$ s.t. $\mathcal{O}^{(j_i)}_{p_i}$ is good at
$p_i$.  Hence by (4.46) there is a $J \in \{ 1 , 2 , \ldots , k ( m )
\}$ such that $\mathcal{O}^{(J)}_{p_i}$ is good for each $i = 1 , 2 ,
\ldots , \ell$.  Hence $f$ is weakly primitive for 
$\mathcal{O}^{(J)} = \Gamma \, v_J$.

\medskip
\noindent
\underline{\large \sf \S5. \ Zariski Density of Prime Matrices}

Fix $n \geq 3$.  We say that an $n \times n$ integral matrix is ``prime''
if all of its coordinates are prime numbers. For $m$ an integer 
$V_{m,n}$ denotes the affine variety given by $\{ x\in Mat_n(\mathbb{R} \, ;\,
\, \det x = m \}$.  We are interested in the set of prime matrices being
Zariski dense in $V_{m,n}$.  For this to happen we must clearly allow
$x$ to have all its coordinates $(x_{ij})$ to be odd numbers.  Such a
matrix $x$ satisfies $\det x \equiv 0 ( 2^{n - 1})$.  It turns out that
this is the only obstruction to producing many primes in $V_{m,n} (
\mathbb{Z})$.  As an application of Vinogradov's methods for analyzing
linear equations in primes with three or more variables, we show 
\setcounter{section}{5}
\setcounter{Thm}{0}
\begin{Thm}
Fix $n \geq 3$.  Then the set of prime matrices $x$ in $V_{m,n} ( \mathbb{Z})$ is
Zariski dense in $V_{m,n}$ iff $m \equiv 0 ( 2^{n-1})$.
\end{Thm}

The proof of Theorem 5.1 can be extended to prove a special case of the
general local to global conjectures for primes in orbits of actions of
certain groups [B-G-S2].

To state the result, let $\Lambda$ be a finite index subgroup of $SL_n (
\mathbb{Z})$, $n \geq 3$.  For $A$ an $n \times n$ integral matrix with
$\det A = m \ne 0$, let $\mathcal{O}_A$ denote the $\Lambda$-orbit
$\Lambda . A.$ Thus $\mathcal{O}_A$ is contained in $V_{m,n} ( \mathbb{Z})$
and is Zariski-dense in $V_{m,n}$.
\begin{Thm}
The set of prime matrices $x$ in $\mathcal{O}_A$ is Zariski dense in $V_{m,n}$ iff
there are no local congruence obstructions.
\end{Thm}

\noindent
{\bf Remark 5.3.}
\begin{enumerate}
\item[(A)] We are using here that for $n \geq 3$ every finite index subgroup of
$SL_n ( \mathbb{Z})$ is a congruence subgroup ([Me], [B-M-S]).

\item[(B)] The general orbit conjecture for this action asserts that
Theorem 5.3 holds for $\Lambda$ a subgroup of $SL_n ( \mathbb{Z})$ which
is Zariski dense in $SL_n$ and with the coordinate functions $x_{ij}$,
$i = 1 , \ldots , n , \, j= 1 , \ldots , n$ replaced by any set $f_1 ,
\ldots , f_t$ of primes in the coordinate ring $\mathbb{Q} [ x_{ij}] \
\slash ( \det x_{ij} - m )$.  In this setting the local congruence
obstructions that need to be passed are that for any $q \geq 2$ there is
an $x$ in $\mathcal{O}_A (\mbox{mod } q)$, the reduction of $\mathcal{O}_A$
modulo $q$, such that $f_1 ( x ) f_2 ( x ) \ldots f_t ( x ) \in (
\mathbb{Z} \slash q \mathbb{Z} )^\ast$.
\end{enumerate}

An example of an orbit in Theorem 5.2 for which there are no local
obstructions for any $\Gamma$ is
\[
\mathcal{O} \, = \, \Gamma
\left[
\begin{array}{rrlcrl}
1 & 1 &&\cdots & 1 & 1 \\
- 1 & 1 && \cdots & 1 & 1 \\
- 1 & - 1 & 1 & \cdots & 1 & 1 \\
- 1 & - 1 &&\cdots& - 1 & 1
\end{array}
\right]
\subset V_{2^{n-1},n} \, .
\]
This is similar to $a = 1$ (or -1) in Dirichlet's theorem, i.e., there are 
infinitely many $p \equiv 1 ( \mod q )$ for any $q$.

\noindent
{\it Proof of necessity of the congruence condition in Theorem 5.1} : If the set of matrices with prime entries is Zariski dense, then of course the set of matrices with odd entries is Zariski dense. But then if $x$ is $n \times n$ integral and has odd entries then writing the
columns of $x$ as $a_1 , \ldots , a_n$, we have: $\det x = \det [a_1 ,
\ldots , a_n] = \det [a_1 , a_2 - a_1 , a_3 - a_1 , \ldots , a_n - a_1]
= \det [ a_1 , 2b_2 , \ldots , 2b_n]$, with $b_j$ integral.  Hence $\det
x =2^{n - 1} \det [a_1 , b_2 , \ldots , b_n] \equiv 0 ( 2^{n-1})$.\qed

To demonstrate the sufficiency of the congruence condition in Theorem 5.1, we will consider the simplest case when $m = 2^{n-1}$. In general
one needs to impose further congruence conditions in the construction
below.
\begin{Lem}
For $n \geq 2$ let
\[
\mathcal{Y} \, = \,
\left\{
\left[
\begin{array}{lcl}
x_{21}x_{22} & \cdots & x_{2n} \\
x_{31} & \cdots & x_{3n} \\
\vdots && \\
x_{n1} & \cdots & x_{nn}
\end{array}
\right]
\right\}
\]
which we identify with affine $A^{(n - 1).n}$ space. 
We denote by  $A_j ( y )$  the $(n - 1 )
\times ( n - 1)$ minor of $y$ gotten by striking the $j$-th column.
 Let
$\mathcal{G}$ be the set of $y \in \mathcal{Y}$ for which
\[
(A_1 ( y ) , A_2 ( y ) , \ldots , A_n ( y )) \, = \, 2^{n-2} \tag{i}
\] 
\[ A_1 ( y ) A_2 ( y ) \ldots A_n ( y ) \ne 0 \tag{ii}
\]
\[
A_1 ( y ) \, + \, A_2 ( y ) \, + \cdots \, + A_n ( y ) \, \equiv \, 0
(2^{n - 1} ) \tag{iii}
\]
and the $x_{ij} $ (where $( x_{ij} ) = y)$ are all prime.  Then $\mathcal{G}$ is
Zariski dense in $\mathcal{Y}$. \end{Lem}

\begin{proof}
We use Dirichlet's theorem repeatedly and proceed by induction on $n$.

For $n = 2$, $y = [ x_{21}, x_{22}]$ and we seek $(x_{21}, x_{22} ) = 1
, x_{21} + x_{22} \equiv 0 ( 2 )$ and $x_{21} , x_{22}$ both prime.
Clearly the set $\mathcal{G}$ of such is Zariski dense in
$A^2$.

For $n \geq 3$ we assume the Lemma for $n - 1$ and construct the $y \in
\mathcal{G}$ as follows:

For $z\, = \, 
\left[
\begin{array}{l}
z_2 \\
\vdots \\
z_n 
\end{array}
\right] \, ,  
\, \xi \, = \,
\left[
\begin{array}{c}
\xi_2 \\
\vdots \\
\xi_n
\end{array}
\right]
$, and $w$ an $(n - 1) \times ( n - 2)$ matrix we write
\[
y \, = \, [ z \, \xi \, w ] \, .
\]

By induction, the set of $w$'s in the space $\mathcal{W}$ of such $(n -
1) \times (n - 2)$ matrices for which $w_{ij}$ are all prime and such
that 
\setcounter{equation}{0}
\begin{equation}
C_2 ( w ) \, + \, C_3 ( w ) \, \cdots \, + \, C_n ( w ) \, \equiv \,
0 ( 2^{n - 2} )
\end{equation}
and
\begin{equation}
(C_2 ( w ) , \ldots , C_n ( w )) \, = \, 2^{n - 3}
\end{equation}
is Zariski dense in $\mathcal{W}$.  Here $C_i ( w )$ is the $(n - 2)$
minor of $w$ obtained by striking the $i$-th row of $w$.  For such a $w$
we seek $\xi$ satisfying 
\begin{equation}
A_1 = \xi_2 C_2 - \xi_3 C_3 \, \cdots \, + \, ( - 1)^n \, \xi_n C_n \,
\equiv \, 2^{n-2} ( 2^{n-1} )
\end{equation}
and 
\begin{equation}
( \xi_j , 2 ) \, = \, 1 \, .
\end{equation}
In view of (5.1) and (5.2), this amounts to 
\begin{equation}
\xi_2 C^\prime_2 - \xi_3 C^\prime_3 \cdots + ( - 1)^n \xi_n C^\prime_n \equiv 2 ( \mbox{mod}
4)
\end{equation}
where
\begin{equation}
(C^\prime_2 , \ldots , C^\prime_n ) \, = \, 1 \hspace{.9em} \mbox{and}
\hspace{.9em} C^\prime_2 \, + \, \cdots C^\prime_j + \cdots C^\prime_n \, \equiv \, 0 (
2 ) \, .
\end{equation}

\medskip
\noindent
According to (5.6) the number  $\ell$ of $C^\prime_j$'s which are $\equiv \pm
1(4)$ is even and positive.  Collecting these $C^\prime_j$'s on the left,
renumbering the indices and replacing $\xi_j$ by $- \xi_j$ suitably,
leads to solving 
\begin{equation}
\xi_2 + \xi_3 + \cdots \xi_{\ell + 1} \, 
\equiv \, b ( \!\! \!\!\!\! \mod 4)
\end{equation}

\noindent
where $b$ is either $0$ or $2 \!\! \mod 4$.  If $b \equiv 0 ( 4 )$ choose
$\xi_j = ( - 1 )^j$, $2 \leq j \leq \ell + 1$ while if $b \equiv 2 ( 4 )$
choose $\xi_2 = \xi_3 = 1$ and $\xi_j = (-1)^j$ for $4 \leq j \leq
\ell + 1$.  Since $\ell$ is even these choices solve (5.5).

Having found such a $\xi (\!\!\!\! \mod 2^{n - 1})$ satisfying (5.3) and (5.4) we choose
$\xi$ integral satisfying this congruence and for which $\xi_j$ are all
prime.  This is possible by Dirichlet's theorem and their choice is
Zariski dense in $\xi$ space.

For each choice of $w$ and $\xi$ above we choose $z$ as follows, First,  we
have
\begin{equation}
A_1 \, \equiv \, 2^{n-2} (\!\!\!\!\!\! \mod 2^{n-1} )
\end{equation}
and hence $A_1 \ne 0$.  For each odd prime $p$ dividing $A_1$, let
$t^{(p)} = \, \left[
\begin{array}{c}
t^{(p)}_2 \\
\vdots \\
t^{(p)}_n
\end{array}
\right] $ 
be chosen with
$t^{(p)}_j \in ( \mathbb{Z} \slash p \mathbb{Z})^\ast$ and satisfying 
\begin{equation}
A_2 := t^{(p)}_2 C_2 - t^{(p)}_3 C_3 \cdots + ( - 1 )^n \, t^{(p)}_n 
C_n \not\equiv \, 0 ( p ) \, .
\end{equation}
It is clear that such a $t^{(p)}$ can be found since $( C_2 , \ldots , C_n
) = 2^{n - 3}$ and $p \geq 3$.

Next let $q_3 , \ldots, q_n$ be distinct primes different from 2 and any
prime divisor of $A_1$ and of 
any entry of $w$.  We chose $z$ to satisfy the following congruences:
\begin{equation}
\left[
\begin{array}{c}
z_2 \\
\vdots \\
z_n
\end{array}
\right] \, \equiv \,
\left[
\begin{array}{c}
w_{2j} \\
\vdots \\
w_{nj}
\end{array}
\right]
(\!\!\!\!\!\!\mod q_j) \, , \hspace{.9em} \mbox{for} \hspace{.9em} 3 \, \leq
\, j \, \leq n \, .
\end{equation}
\begin{equation}
\hspace{-.15in}
\left[
\begin{array}{c}
z_2 \\
\vdots \\
z_n
\end{array}
\right]
\, \equiv \, 
\left[
\begin{array}{c}
t^{(p)}_2 \\
\vdots \\
t^{(p)}_n
\end{array}
\right]
(\!\!\!\!\!\!\mod p) \hspace{.9em} \mbox{for} \hspace{.9em} p | A_1 \, ,
\, p \ \mbox{odd}
\end{equation}
\begin{equation}
\hspace{-1.35in}
\left[
\begin{array}{c}
z_2 \\
\vdots \\
z_n 
\end{array}
\right] \, \equiv \, 
\left[
\begin{array}{c}
\xi_2 \\
\vdots \\
\xi_n
\end{array}
\right]
(\!\!\!\!\!\!\mod 2^{n-1})
\end{equation}

The conditions (5.10), (5.11), (5.12) involve distinct prime moduli and the
numbers on the right are all prime to their moduli, hence by Dirichlet's theorem 
we can choose $z_j$ to be prime and to satisfy the congruences 
(5.10), (5.11), (5.12).  Moreover, 
the set of choices for these $z$'s is Zariski dense in
the space of $z$'s.  This produces matrices $y = [z \xi w]$ which we
check satisfy the requirement of the lemma.  Indeed by (5.8) 
\begin{equation}
A_1 ( y ) \, \equiv \, 2^{n - 2} (\!\!\!\!\!\!\mod 2^{n - 1}) \, .
\end{equation}
By (5.9) $A_1 ( y ) $ and $ A_2 ( y )$ have no odd prime common factor.  Also by (5.12)
\begin{equation}
A_2 ( y ) \, \equiv \, A_1 ( y ) (\!\!\!\!\!\!\mod 2^{n-2})
\end{equation}
so we conclude that 
\begin{equation}
(A_1 ( y ) , \, A_2 ( y )) \, = \, 2^{n-2} \, .
\end{equation}
Note that for $3 \leq j \leq n$,   $ A_j ( y ) \equiv 0 ( 2^{n - 2})$ since $A_j$ is the
determinant of an $( n - 1) \times ( n - 1)$ matrix with odd entries.
Hence
\begin{equation}
A_1 ( y ) + A_2 ( y ) + \cdots + A_n ( y ) \, \equiv \, 0 ( \!\!\!\!\mod
2^{n-1} ) \, .
\end{equation}

Thus together with (5.13) we deduce that 
\begin{equation}
(A_1 ( y ) , \ldots , A_n ( y )) \, = \, 2^{n-2}
\end{equation}
From (5.13) and (5.14) we conclude that $A_1 ( y ) A_2 ( y ) \ne 0$, while from (5.10) we have that $A_j ( y ) \equiv \pm A_1 ( y )
(\!\!\!\!\mod q_j)$ for $3 \leq j \leq n$ and hence $A_3 ( y )
\ldots A_n ( y ) \ne 0$.  All this, coupled with the fact that the
entries of $y$ are prime and that the $y$'s can be chosen to be Zariski
dense in $\mathcal{Y}$, completes the proof of Lemma 5.4. 
\end{proof}

We now appeal to Vinogradov's methods [Vi] for studying the solvability
of 
\begin{equation}
H_{A_0 , \ldots , A_n}: \, A_1 s_1 - A_2 s_2 \cdots + ( - 1)^{n+1} \,
A_n s_n - A_0 = 0
\end{equation}
with $s_j$ prime (i.e., $(s_j)$ a prime ideal in $\mathbb{Z}$).  The
treatment in Vaughan [V], pp. 37 shows that if $n \geq 3$ and $A_1 A_2
\ldots A_n \ne 0$, then the number of solutions to (5.18) with $| s_j|
\leq T$ and $s_j$ prime satisfies
\begin{equation}
R ( T ) \gg \, C \, 
\frac{T^{n-1}}{(\log T)^{n}} \, + \, O_\nu \, \left(
\frac{T^{n-1}}{( \log T )^\nu} \right)
\end{equation}
for any fixed (large) $\nu$.  Moreover, the critical number $C$ given by the
singular series is nonzero {\sf iff} the following local conditions are
satisfied
\begin{equation}
\begin{array}{lll}
(A_0 , A_1, \ldots , A_{n-1} ) & = & (A_0, A_1 , \ldots , A_{n-2} , A_n
) \,
\\
& = & (A_1 , A_2 , \ldots , A_n ) \, = \, (A_0 , A_1 , \ldots , A_n)
\end{array}
\end{equation}
and
\begin{equation}
A_0 + A_1 + \cdots + A_n \, \equiv \, 0 (\!\!\!\!\!\!\mod 2 ( A_0 , A_1 ,
\ldots , A_n )) \, .
\end{equation}

If any of these conditions fail, for example if there is a prime $p$ and
an $e \geq 1$ with $p^e | A_j$ for $j = 0, \ldots , n - 1$ but $p^e
\nmid A_n$, then any solution to (5.18) must have $p | s_n$.  Hence the
set of solutions to (5.18) with $s_n$ prime is not Zariski dense in
$H_{A_0 , \ldots , A_n}$.  Thus the conditions (5.20), (5.21) are
necessary for the Zariski density of $(s_1 , \ldots , s_n ) , s_j$ prime
in $H_{A_0 , \ldots , A_n}$.  These conditions are also sufficient.
Indeed $H_{A_0 , \ldots , A_n}$ is connected and hence if these points
are not Zariski dense then there is a polynomial $f ( s_1 , \ldots ,
s_n)$ which is nonconstant on $H_{A_0 , \ldots , A_n}$ such that all the
$s$'s lie in $H_{A_0 , \ldots , A_n} \cap \{s : f ( s ) = 0 \}$.
It is elementary that the number of integer points in this intersection
and for which $| s_j| \leq T$ is $O ( T^{n - 2})$.  Hence if $C \ne 0$
then (5.19) gives a contradiction to the points all lying in $\{ s: f (
s ) = 0 \} \cap H_A$.  We conclude that 
\begin{equation}
\{ ( s_1 , \ldots , s_n ): s_j \ 
\mbox{is prime}, \ \ s \in H_{A_0 ,
\ldots , A_n} \}
\end{equation} 
is Zariski dense in $H_{A_0 , \ldots , A_n}$ {\sf iff} (5.20) and (5.21)
hold.

Let $\mathcal{Y}$ be the space in Lemma 5.4 and $\mathcal{G}$ the set of
$y$'s constructed in that Lemma.  Set $A_0 = 2^{n-1}$.  Then for $y \in
\mathcal{G}$
\[ (A_1 ( y ) , \ldots , A_n ( y )) \, = \, 2^{n-2}
\]
and
\[
A_1 ( y ) + \cdots + A_n ( y ) \, \equiv \, 0 ( 2^{n-1})
\]
and $A_1 ( y ) \ldots A_n ( y ) \ne 0$.  Hence $(A_0 , A_1 ( y ) ,
\ldots , A_n ( y )) = 2^{n-2}$ and the number of $1 \leq j \leq n$ for
which $2^{n - 2} | A_j ( y )$ is even and positive.  It follows
that
\[ (A_0 , A_1 ( y ) , \ldots , A_{n-1} ( y )) \, = \, (A_0 , A_1 ( y ) ,
\ldots , A_{n - 2} ( y ) , A_n ( y )) \, = 
\]
\[
= \, (A_1 ( y ) , \ldots , A_n(y )) \, = \, 2^{n-2} \hspace{.25in} \mbox{and}
\]
\[
A_0 \, + \, A_1 ( y ) \cdots \, + \, A_n ( y ) \, \equiv \, 0 \!\!\! \! \mod (2
(A_0 , \ldots , A_n ( y )))
\, . \]
Thus (5.20) and (5.21) are satisfied and hence by (5.22) it follows that
for any $y \in \mathcal{G}$ the set of $s \in H_{A_0 , \ldots , A_n ( y
)}$ for which all $s_j$ are prime is Zariski dense in $H_{A_0 , \ldots , A_n ( y
)}$.  For each such $y$ and $s$ the matrix $x$

\[
\left[
\begin{array}{lcl}
s_1 & \cdots & s_n \\
& y & 
\end{array}
\right] \, = \,
\left[
\begin{array}{llll}
x_1 & x_2 & \cdots & x_n \\
z_2 & \xi_2 && \\
\vdots & \vdots & w & \\
z_n & \xi_n &&
\end{array}
\right] \, , \hspace{.9em} \mbox{is a prime matrix in} \ V_{ 2^{n-1},n}
\, . 
\]
To complete the proof of Theorem 5.1 with $m = 2^{n - 1}$, we must show
that set of $x$'s constructed above is dense in $V_{ 2^{n - 1},n}$.
Let $\mathcal{Y}^0 = \{ y \in \mathcal{Y}: A_j ( y ) \ne 0$ for some $1
\leq j \leq n \}$.  $\mathcal{Y}^0$ is an open irreducible subset of
$A^{(n - 1 ) \times n}$ and is quasi-affine.  Let $\Upsilon :
V_{m,n} \rightarrow \mathcal{Y}^0$ be the surjective morphism 
\begin{equation}
x \, \longrightarrow \,
\left(
\begin{array}{l c l}
x_{21} & \cdots & x_{2n} \\
\vdots && \\
x_{n 1} & \cdots & x_{nm}
\end{array}
\right)
\end{equation}
If $U$ is a nonempty open subset of $V_{m,n}$ then since $V_{m,n}$ is
connected, $U$ is dense and hence $\Upsilon ( U )$ is dense in $\Upsilon ( V_{m,n} ) = \mathcal{Y}^0$.  Also $\Upsilon ( U )$ is constructible and contains
an open dense subset of $\mathcal{Y}^0$.  According to Lemma 5.4 there
is a $y \in \mathcal{G} \cap \Upsilon ( U )$.  
Now $U \cap \Upsilon^{-1} ( y )$ is
a nonempty open subset of $H_{2^{n-1} , A_1 ( y ) , \ldots , A_n ( y)}$.
According to the analysis above it contains a point $(p_1, \ldots ,
p_n)$ all of whose coordinates are prime.  Hence $x = \left[
\begin{array}{c} p \\ y \end{array} \right]$ is a prime matrix and it is in
$U$.  This proves Theorem 5.1.

\medskip
\noindent
\underline{\large \sf Section 6. \ Spectral estimates for uniform lattices}

We now turn to explain an alternative approach to estimating the level
of distribution.  Indeed, rather than giving an error term for the
number of lattice points in a ball, it suffices to estimate the
deviation of a positive weighted sum over the lattice points.  This allows one to
take a smooth weight and to estimate its Fourier transform directly via a
convergent eigenfunction expansion of the corresponding automorphic
kernel.  This method gives a sharper result for the level of
distribution and the improvement is most significant when the lattice is
co-compact.  Since the latter assumption also allows us to avoid the
analysis of Eisenstein series necessary for the estimates of
eigenfunction expansions, we will present the method only for co-compact
lattices.  In this case it can be naturally viewed as providing the
error estimate in the non-Euclidean version of a Poisson summation
formula for compact locally symmetric spaces.  To simplify matters
further, we will assume that the weight functions are radial, allowing
us to reduce matters to spectral estimates associated with spherical
functions.  We note that the estimate of the deviation we give below is
in fact sharp: it gives the best possible result for a smooth weighted
sum. 

We retain the notation of \S\S3.1-3.3, and again let $G$ be a connected
semisimple Lie group with finite center and no compact factors.  $\mathcal{S}$
denotes the symmetric space $\mathcal{S} = G/K$ where $K$ is a maximal compact
subgroup.  We take the Riemannian structure on $\mathcal{S}$ induced by the
Cartan-Killing form on $G$ and let $d$ denote the associated
$G$-invariant distance.   Consider the family of kernels on $\mathcal{S} \times \mathcal{S}$
given by:
\[
L_t ( z , w ) \, = \, \chi_{[0,t]} \, (d ( z , w )) 
\]
where $\chi_{[0,t]}$ is the characteristic function of an interval.  Fix
a smooth function $b ( w )$ on $\mathcal{S}$ which is non-negative, positive
definite, supported in a ball of radius $t_0$ with center $w_0$,
invariant under $K_{w_0}$ and with integral 1.  Define the following
smooth function, which is supported on the set of points whose distance
from $w_0$ is at most $t+t_0$:

\[
W_t ( z ) \, = \, \int\limits_{\mathcal{S}} \, L_t ( z , w ) \, b ( w ) \, d \,
{\rm vol} ( w ) \,  \tag{6.0}
\]
Let $\Gamma$ be any  uniform lattice  in $G$  which satisfies that
for any $z$\break $| \{ \gamma \in \Gamma: \, d ( \gamma z , z) \, \leq \, 1
\} | \, \leq \, c$ with $c$ fixed. Our version of the error
estimate in the Poisson summation formula is as follows.
\setcounter{section}{6}
\setcounter{Thm}{0}
\begin{Thm}
{\bf Poisson summation formula.} \ Let $\Gamma$ be a uniform irreducible lattice in
$G$ as above. Then for $\eta > 0$ fixed, 
\[
\sum\limits_{\gamma \in \Gamma} \: W_t ( \gamma z ) \, = \,
\frac{\rm{vol} \, B_t}{{\rm vol} ( \Gamma \backslash G )} \, + \, O_\eta
\, ( ( {\rm vol} \, B_t )^{1-\frac1p + \eta} ) \, , 
\]

\noindent
where 
\begin{enumerate}
\item The result holds uniformly for arbitrary $z$ and $w_0$.
\item  $p = p ( G , \Gamma )$ is the integrability parameter of the representation in $L^2_0(\Gamma\backslash G)$ as in
Theorem 3.3.  In particular $1-\frac1p = \frac{1}{2}$ if and only if all
representations weakly contained in $L^2_0 ( \Gamma \backslash G)$ are tempered.
\end{enumerate}
\end{Thm}
Let us note that it is indeed the case 
that $p(G,\Gamma)< \infty$ for every irreducible lattice and refer to [Ke-Sa] for discussion and references.

The proof is based on establishing uniform control on the pointwise
spectral expansion of the smooth function $W_t$.  We will use the spectral expansion associated 
with the commutative algebra $\mathcal{D}$ of $G$-invariant
differential operators on the symmetric space (recall that the Laplacian generates this algebra if and
only if the real rank of $G$ is one).  Being $G$-invariant, these
differential operators descend to operators on $M = \Gamma \backslash
\mathcal{S}$, and admit a joint spectral resolution. The eigenvalues are given 
by (infinitesimal) characters $\omega_\lambda : \mathcal{D}\to \mathbb{C}$, parametrized  
by $\lambda\in \Sigma\subset \mbox{Hom} (\mathfrak{a},\mathbb{C})/\mathcal{W}$. 
Here $G=NAK$ is an Iwasawa decomposition, $\mathfrak{a}$ is the Lie algebra of $A$, $\mathcal{W}$ is the Weyl group of $(\mathfrak{g},\mathfrak{a})$, and $\Sigma$ 
parametrizes the positive-definite spherical functions $\Psi_\lambda : G\to \mathbb{C}$.  
We let $\norm{\cdot}$ denote the Euclidean norm associated with the inner product on $\mathfrak{a}$ 
given by the restriction of the Killing form. Finally recall that every (normalized) joint eigenfunction $\phi$ on $\Gamma\backslash\mathcal{S}$ of the algebra $\mathcal{D}$ is also a joint eigenfunction of the commutative convolution algebra $L^1(K\backslash G/K)$  of bi-$K$-invariant kernels on $G$. The eigenvalue constitutes a complex homomorphism $\tilde{\omega}_\lambda$ of this algebra, which corresponds uniquely to a spherical function $\Psi_\lambda$. Thus if $F$ is bi-$K$-invariant  
$$\pi_{\Gamma\backslash G}(F)(\phi)=\tilde{\omega}_\lambda (F)\phi \mbox{ , where  } 
\tilde{\omega}_\lambda(F)=\int_G F(g)\Psi_\lambda(g)dm_G(g)\,\,.$$

We begin with the following result:
\begin{prop} {\bf The average size of eigenfunctions.}  Let $\Gamma$ be
any uniform lattice in $G$, satisfying the condition preceding Theorem 6.1,
and let $\phi_j , j \in \mathbb{N}$ be an orthonormal basis for $L^2 ( M
)$ consisting of joint eigenfunctions of the ring $\mathcal{D}$ of $G$-invariant
differential operators on $M$.  Denote by $\Psi_{\lambda_j}$ the spherical function associated with $\phi_j$.  Then there exists a constant $C ( \mathcal{S} )$ 
depending only on $\mathcal{S}$ and independent of $\Gamma$, such that 
\[
\mathop{\max}\limits_{z \in M} \: \sum\limits_{\norm{\lambda_j }\, \leq \,
\norm{\lambda}} \, | \phi_j ( z ) |^2 \, \leq \, C ( \mathcal{S} ) \,(1+\norm{\lambda})^{\dim(\mathcal{S})} \, .
\]
\end{prop}
\begin{proof}
Let $\ell_\epsilon ( z , w )$ be a smooth non-negative positive-definite
kernel on $\mathcal{S} \times \mathcal{S}$, depending only on $d ( z , w )$ and supported in
$d ( z , w ) \leq \epsilon$ with unit integral.  The automorphic kernel
$A_\epsilon ( \Gamma z, \Gamma w ) \, = \, \sum\limits_{\gamma \in
\Gamma} \ell_\epsilon ( \gamma z , w )$ on $M \times M$ can be expanded
in terms of the joint eigenfunctions $\phi_j$ of $\mathcal{D}$ or $L^1(K\backslash G/K)$ as 
\[
A_\epsilon ( \Gamma z , \Gamma w) \, = \,
\sum\limits_{j \, = \, 1}^{\infty} \, h_{A_\epsilon} ( \lambda_j ) \,
\phi_j ( \Gamma z) \, \overline{\phi_j ( \Gamma w )}
\]
where $h_{A_\epsilon} ( \lambda_j )$ are the eigenvalues of the operator
defined on $L^2 ( M )$ by the automorphic kernel $A_\epsilon$.  As noted above, these 
eigenvalues are given by the Selberg-Harish Chandra spherical transform
(normalized at $w = w_0)$
\[ 
h_{A_\epsilon} ( \lambda_j )
 \, = \, \int\limits_{\mathcal{S}} \, \ell_\epsilon (
z , w_0 ) \, \Psi_{\lambda_j} ( z ) d \, {\rm vol} ( z )
\]
where $\Psi_{\lambda_j}$ is the spherical function
associated with  $\phi_{\lambda_j}$, normalized by $\Psi_\lambda(e)=1$, and viewed as a function on $\mathcal{S}$. 

We claim that 
\[ \abs{h_{A_\epsilon} ( \lambda )-1}\,\le\, 
\abs{ \, \int\limits_{z \in B_\epsilon
(w_0)} \, \ell_\epsilon ( z , w_0) \, \abs{\Psi_{\lambda} ( z )-1} d \, {\rm
vol} ( z )} \, \le \,C_1(\mathcal{S})(1+\norm{\lambda})\epsilon \,\,.
\]
Clearly,  this estimate follows from the fact that for all $H$ in the unit sphere in $\mathfrak{a}$ and $\abs{t}\le 1$ (say),  the first derivative of the normalized positive definite spherical functions $\Psi_\lambda$ satisfy :
  $$\abs{\frac{d}{dt}\Psi_\lambda(\exp (tH)}\le C_1(\mathcal{S})(1+\norm{\lambda})\,\,.$$
  This estimate is a consequence of the Harish Chandra power series expansion for the spherical 
  functions, together with the fact that normalized positive definite spherical functions are all bounded by $1$. The estimate follows from e.g.  [G-V, Prop. 4.6.2].

We conclude that if 
$(1+\norm{\lambda} )\epsilon < \frac12 C_1(\mathcal{S})$ then 
$h_{A_\epsilon} ( \lambda) \, \geq \, \frac{1}{2}$, and therefore we obtain the following upper bound on the average size of the
eigenfunctions 
\[
\sum\limits_{\gamma \, \in \, \Gamma} \, \ell_\epsilon ( \gamma z , z )
\, = \,
\sum\limits_{j \, = \, 1}^{\infty} \, h_{A_\epsilon} ( \lambda_j) |
\phi_j ( z )|^2 \, \geq \, \frac{1}{2} \, \sum\limits_{\norm{\lambda_j} \, < \,
C_2(\mathcal{S}) / \epsilon} \, | \phi_j ( z )|^2 \, .
\]

On the other hand, we can clearly obtain a pointwise upper bound of the
form
\[
\sum\limits_{\gamma \in \Gamma} \, \ell_\epsilon ( \gamma z , z ) \,
\leq \, C_3(\mathcal{S})  \epsilon^{- \dim \mathcal{S}} \, .
\]
Indeed, this follows when the kernel is defined by a bump 
function which satisfies the obvious upper bound of being $\ll
\epsilon^{- \dim \mathcal{S}}$, and taking also into account the fact that there
are at most $c$ lattice points in a ball radius $\epsilon\leq 1$, this
coming from our assumption on $\Gamma$.  Combining the two estimates we
can conclude that
\[
\sum\limits_{\lambda_j \, < \, \lambda \, = \, C_2 (\mathcal{S}) / \epsilon} \, |
\phi_j ( z ) |^2 \, \leq \, C_3(\mathcal{S}) \epsilon^{- \dim \mathcal{S}} \, \leq \, C(\mathcal{S})
(1+\norm{\lambda})^{ \dim \mathcal{S}} \, .
\]
The proof of Proposition 6.2 is now complete.
\end{proof}

\noindent
{\it Proof of Theorem 6.1}. \ Consider the identity
\setcounter{equation}{0} 
\begin{equation}
\sum\limits_{\gamma \in \Gamma} \, 
\int\limits_{\mathcal{S}} 
\, L_t ( \gamma z , w )
\, b (  w ) dw \, = \,
\sum\limits_{\gamma \in \Gamma } \, W_t ( \gamma z) \, = \,
\sum\limits_{j \, = \, 0}^{\infty} \, 
h_{L_t} ( \lambda_j ) \, \phi_j ( z
) \,
\int\limits_{\mathcal{S}} \, \overline{\Psi_{\lambda_j} ( w )} \, b ( w ) dw
\end{equation}
The eigenvalue $\lambda_0 = 0$ associated with the constant function
$\phi_0 = 1 / {\rm vol} ( M )$ (the unique $\Delta$-eigenfunction with
this eigenvalue) gives the main contribution to the infinite sum, which
is $\frac{{\rm vol} B_t}{{\rm vol} M} \, = \, h_{L_t} ( 0 )$.  We must
therefore estimate the contribution of all other terms.

Now note that since the bump function $b ( w )$ is a fixed smooth
function, and $\phi_j$ is an eigenfunction of the Laplacian $\Delta$ with
eigenvalue $\omega_{\lambda_{j}}(\Delta), m$ integrations by parts give, for any fixed $m$
and all $j \in \mathbb{N}$
\[
\int\limits_{\mathcal{S}} \, b ( w ) \, \overline{\Psi_{\lambda_j} ( w)} \, d w \,
\leq \, C_{m}(1+ \norm{\lambda_j})^{-m} \, .
\]
Let us denote $\hat{b} ( \lambda_j ) = \int_{\mathcal{S}} \, b ( w ) \,
\overline{\Psi_{\lambda_j} ( w )} \, dw$.  Recall that $\lambda_j\,,\,j\neq 0$ is a discrete set,
 and thus have a fixed positive distance from $0$, due to our spectral gap assumption.
As a consequence, the spherical functions $\Psi_{\lambda_j}$, all have a
fixed rate of decay, which can be expressed as a negative power of the
volume of $B_t$.

Now
\[
h_{L_t} ( \lambda ) \, = \, \int\limits_{\mathcal{S} }\, L_t ( z , w_0 ) \,
\overline{\Psi_\lambda ( z )} \, dz
\]
is the averages of the spherical function on a ball of radius $t$ and
center $w_0$, which by H\"older's inequality is estimated by
${\rm vol} \,B^{\delta}_t$, $\delta=1-\frac1p+\eta$. 

Therefore using (6.1) we can write

\[
\Bigg| \sum\limits_{\gamma \, \in \, \Gamma} \,
W_t ( \gamma z ) \, - \,
\frac{ {\rm vol} \, B_t}{{\rm vol} \, M} 
\Bigg| \, \leq \, ( {\rm vol} \, B_t)^\delta \, 
\sum\limits_{j \, \ne \, 0} \, | \phi_j ( z ) |
\Big|\hat{b} ( \lambda_j ) \Big| \, .
\]
Now $| \hat{b} ( \lambda_j ) | \, \leq \, C_m (1+\norm{\lambda_j})^{-m}$, and by Lemma 6.2  
$\sum_{\lambda_j \, \leq \, \lambda} | \phi_j ( z
) |^2 \, \leq \, C(\mathcal{S})(1+\norm{\lambda})^{\dim \mathcal{S}}$, so that 
upon choosing $k$ large enough
\[
\sum\limits_{j \, \ne \, 0} \, | \phi_j ( z ) | \,
\big|\hat{b} ( \lambda_j ) \big| \, \leq \,
\left(
\sum\limits_{j \, \ne \, 0} \,(1+\norm{ \lambda_j})^{-2k}\, | \phi_j ( z )|^2
\right)^{1/2} \,
\left( \sum\limits_{j \, \ne \, 0} \, (1+\norm{\lambda_j})^{2k} \, \big| \hat{b} (
\lambda_j ) \big|^2 \right)^{1/2} \, < \, \infty \, .
\]
This concludes the proof of Theorem 6.1. \hfill $\square$

\medskip
\noindent
{\bf Remark 6.4.}
The error estimate in the Poisson summation formula can be similarly established when $\Gamma$ is non-uniform, using the foregoing arguments and the 
theory of Eisenstein series.

Next we apply Theorem 6.1 in the context of sieving as in Sections 2,3
and 4.  For the purpose of the lower bound sieve we can use the
nonnegative weight function $W_t$ in (6.0).  Using our previous setup
and notations, let us work with the distance parameter $T=e^t$, where $t$ denotes the distance is the symmetric space $\mathcal{S}$. 
But notice that since we are now working with symmetric 
space distance and not with a norm, in general the exponent of volume growth is now 
$2\norm{\rho_G}$, namely the rate of volume growth for Riemannian balls in $\mathcal{S}$. 
Recall also that $t(f)$ denotes the number of irreducible factors of the polynomial $f$. Now consider 
\begin{equation}
S_{W_T} ( \mathcal{A} , P) : = \, \sum\limits_{\gamma \, \in \, \Gamma\atop{( f (
\gamma v ) , P ) \, = \, 1}} \, W_T ( \gamma ) \, .
\end{equation}
where $W_T ( \gamma ) := \, W_T ( \gamma z_0)$ for a fixed $z_0 \in
\Gamma \backslash G \slash K$.  

We have
\[
0 \, \leq \, W_T ( \gamma ) \, \leq \, 1
\]
and for $T$ large 
\begin{equation}
W_T ( \gamma ) \, = \, 
\left\{
\begin{array}{lll}
1 & {\rm if} & \parallel \gamma \parallel \, \leq \, \frac{T}{2} \\
\\
0 & {\rm for} & \parallel \gamma \parallel \, \geq \, 2T
\end{array}
\right.
\end{equation}
where $\parallel \, \parallel$ is a bi-$K$-invariant norm on $G$.

Theorem 6.1 (assuming as we do from now on that $\Gamma$ is co-compact)
gives the conclusion that uniformly for $y \in \Gamma$
\begin{equation}
\frac{1}{{\rm vol} \, ( B_T)} \,
\sum\limits_{\gamma \in \Gamma ( q )} \, W_T ( \gamma y) \, = \,
\frac{1}{[ \Gamma : \Gamma ( q ) ]} \, +\, O_\epsilon \,
\left(
T^{- \frac{a}{p} \, + \, \epsilon} \right) \, .
\end{equation}

This corresponds to Theorem 3.2 with $\theta \slash{( 1 + \dim G)} \, =
\, {a}\slash{2n_e ( 1 + \dim G)}$ replaced by $a / p$, where $p=p(G.\Gamma)$.  Running the rest of
the sieving analysis with this positive smooth weight $W_T$ to the end
of Section 4 yields an improvement in Theorem 1.7 and Corollary 1.8 with
the condition on $r$ replaced by 
\begin{equation}
r \, > \, \frac{9 . t(f) . \, p(G,\Gamma). ( \dim G ) . \, \mbox{deg }f} {a} \, .
\end{equation}

Note that we have incorporated the small improvement (4.26) of (4.25) as
well.  

To further improve this value of $r$ we use the weighted sieve
([H-R] Chapter 10) in
place of the elementary sieve in Section 2.  The form which is
convenient for us is as follows:

Let $a_k \, \geq 0$ be a finite sequence and assume that for $d \geq 1$
\begin{equation}
\sum\limits_{k \, \equiv \, 0 ( d )} \, a_k \, = \, \frac{\rho ( d )}{d}
\, X + R ( \mathcal{A} , d )
\end{equation}
with $R ( \mathcal{A }, 1) = 0$, $\rho ( 1 ) = 1$ and $\rho$ multiplicative and
that $\rho(p)$ satisfies 2.5 for all $p \geq 2$.  Concerning the sieve
dimension $t$ assume that for $2 \leq z_1 \leq z$ we have
\begin{equation}
\mathop{\Pi}\limits_{z_1 \, \leq \, p \, < \, z} \,
\left(
1 \, - \, \frac{\rho ( p )}{p} \right)^{-1} \, \leq \,
\left( \frac{\log z}{\log z_1} \right)^t \, \left(
1 \, + \, \frac{A}{\log z_1} \right)
\end{equation}
for some fixed constant $A$.

Assume that we have a level  distribution $\tau$, that is for
$\epsilon > 0$
\begin{equation}
\sum\limits_{d \, \leq \, X^\tau} \, | R ( \mathcal{A} , d ) | \,
\mathop{\ll}\limits_{\epsilon} \, X^{1 - \epsilon} \, .
\end{equation}
Define $\mu$ by
\begin{equation}
\mathop{\max}\limits_{a_n \, \in \, \mathcal{A}} \, n \, \leq \, X^{\tau
\mu} \, .
\end{equation}
Let $P_r$ denotes the set of positive integers with at most
$r$-prime factors. Then
for any $0 < \rho < \nu_t$ and
\begin{equation}
r \, > \, \left(
1 \, + \, \rho - \frac{\rho}{\nu_t} \right) \, \mu - 1 + ( t + \rho ) \,
\log \, \frac{\nu_t}{\rho} - t + \frac{\rho t}{\nu_t} \,, 
\end{equation}
there is $\delta = \delta ( t , \mu , r , \rho ) > 0$ such that 
\begin{equation}
\displaystyle{\sum\limits_{k \in P_r}} \, a_k \, \geq \, \delta \,
\frac{X}{( \log X)^t} \, .
\end{equation}
Here $\nu_t$ is the sieve limit in dimension $t$, see [H-R] for a table
and for the fact that $\nu_t \leq 4t$ which is what we will use.

We apply this to our sequence
\begin{equation}
a_k (  T ) \, = \, \sum\limits_{| f ( \gamma v ) | \, = \, k } \, W_T (
\gamma) \, .
\end{equation}
By (6.4) and the analysis in Section 4 we have
\begin{equation}
\tau \, = \, \frac{1}{p \, \dim G}
\end{equation}
\begin{equation}
\mu \, = \, \frac{p \, (\dim G) \, \deg(f)}{a} \, .
\end{equation}
Taking $\zeta = 1$ in (6.10) for simplicity leads to 6.11 holding for 
\begin{equation}
r \, > \, \frac{2 p ( \dim G ) \, \deg(f)}{a} \, - 1 \, + \, ( t(f) + 1) \, \log
( 4t(f)) \, - t + \, \frac{1}{4} \, .
\end{equation}
In particular Theorems 1.7 and Corollary 1.8 are valid for such $r$.

As noted in Remark 1 of 6.3, these considerations also apply to the case
$G \slash \Gamma$ non-compact.  In particular to $\Gamma = SL_n (
\mathbb{Z})$ and to $V_{n , m} ( \mathbb{Z})$.  In this case $p = 2 ( n
- 1)$ for $n \geq 3$ and it is estimated in (6.17) for $n = 2$, and $a =
  n ( n - 1)$, so that Theorem 1.2 and Corollary 1.3 are valid with 
\begin{equation}
r \, > \, 4 n\deg(f) \, - 1 \, + \, ( t(f) + 1) \, \log ( 4t(f)) \, - t(f) \, + \,
\frac{1}{4} \, ,
\end{equation}
for $n > 2$.

Our final improvement comes in the cases where much stronger bounds
towards the Ramanujan Conjectures are valid especially with $n$ large.
Let $D \slash \mathbb{Q}$ be a division algebra of degree $n$ which for
the reasons below, we assume is itself prime.  Assume that $D \otimes
\mathbb{R} \cong \, {\rm Mat}_{n \times n} ( \mathbb{R})$
and let $N_r$ denote the reduced norm on $D$.  Let $V_{m , D} = \{ x \in
D ( \mathbb{Z}) : N_r ( x ) = m \}$ with $m \ne 0$. Here the
$\mathbb{Z}$ structure is given by the defining equations of $D \slash
\mathbb{Q}$ in $A^N$, $N = n^2$,  $G = \{ x: N_r ( x ) = 1 \}$ and
$\Gamma$ the integral elements of reduced norm equal to 1.  These act on
$V_{D , m}$ making it into a principal homogeneous space.  Let $f \in
\mathbb{Q} [ x_{ij} ]$ which is primitive on $V_{D , m} ( \mathbb{Z} )$.
The discussion of this section applies to the question of the saturation
number $r_0 ( V_{D , m} ( \mathbb{Z} ) , f )$.  What is pleasant about
such compact quotients $G ( \mathbb{R}) \slash \Gamma (q)$ coming from
these division algebras is that we have very good upper bounds for their
corresponding $p$'s.  Specifically any representation $\pi$ occuring in
$L^2_0 ( G ( \mathbb{R}) \slash \Gamma ( q ) )$ corresponds to an
automorphic representation occuring in $L^2 ( D ( \mathbb{A}) \slash D (
\mathbb{Q}))$ which in turn via the Jacquet-Langlands correspondence
[J-L] if $n = 2$ and Arthur-Clozel [A-C] if $n > 2$, lifts to an
automorphic cuspidal representation $\pi$ of $GL_n ( \mathbb{A}) \slash
GL_n ( \mathbb{Q})$ (it is here that we assume that $n$ is prime so that
$\pi$ is not a residual Eisenstein series [M-W]).  Applying the best
known bounds towards the Remanujan Conjectures ``at infinity'' (see [Sa1]
for a survey ) for such $\pi$, we conclude that
\begin{equation}
\begin{array}{lllll}
p ( K \backslash G ( \mathbb{R} ) \slash \Gamma ( q ) ) & \leq &
\frac{64}{25} & \mbox{if} & n \, = \, 2 \, , \\
\\
p (K \backslash G ( \mathbb{R}) \slash \Gamma ( q ) ) & \leq &
\frac{28}{9} & \mbox{if} & n \, = \, 3 \, , \\
\\
p ( K \backslash G ( \mathbb{R}) \slash \Gamma ( q ) ) & \leq &
\frac{2n}{n-2} & \mbox{if} & n \, \geq 4 \ \mbox{is even} \\
\\
p ( K \backslash G ( \mathbb{R}) \slash \Gamma ( q ) ) & \leq & \frac{2
( n + 1)}{n - 1} & \mbox{if} & n \, \geq 5 \ \mbox{ is odd} \, .
\end{array} 
\end{equation}

These follow from (24) (22) and (13) in [Sa1] by computing $p$ using
Theorem 8.48 in [Kn].  For $V_{D , m} ( \mathbb{Z})$ as with $V_{n ,
m}$, $a = n^2 - n$ and $\dim G = n^2 -1$ hence, 6.15 leads to 
\setcounter{Thm}{2}
\begin{Thm}
Let $V_{D , m} ( \mathbb{Z} ) \subset A^N$ be the set of integral points
of norm $m$ in $D$, which we assume is nonempty.  Let $f \in \mathbb{Q}
[ x_{ij}]$ be of degree $d$ and assume that $f$ factors into $t$ irreducible factors in
the coordinate ring $\bar{\mathbb{Q}} [ V_{D, m}]$ and that $f$ is $V_{D , m}
( \mathbb{Z})$-weakly-primitive.  Then
\[
r_0 ( V_{D , m} ( \mathbb{Z}) , f ) \, \leq \, 2 p_n \, \frac{n + 1}{n}
\, d \, + \, ( t + 1) \, \log ( 4t) - t
\]
where $p_n$ are given in 6.17.
\end{Thm}

\vfill
\today:gpp
\end{document}